\RequirePackage{tikz}
\documentclass[pdflatex,sn-mathphys]{sn-jnl}

\usepackage{subfig}
\usepackage{tikz}
\usetikzlibrary{decorations.markings}
\usetikzlibrary{calc,positioning,decorations.pathmorphing,decorations.pathreplacing,fixedpointarithmetic}
\usepackage{figstyle}
\usepackage{program}

\jyear{2021}%

\newtheorem{theorem}{Theorem}
\newtheorem{lemma}[theorem]{Lemma}
\newtheorem{conjecture}[theorem]{Conjecture}

\newcommand{\D}{{\mathfrak{D}}}
\newcommand{\sB}{{\mathfrak{B}}}

\newcommand{\sP}{\mathcal{P}}

\begin{document}

\title[$\alpha$-diperfect and BE-diperfect digraphs]{Some results on Berge's conjecture and Begin-End conjecture}


\author*[1]{\fnm{Lucas Ismaily Bezerra} \sur{Freitas}}\email{ismailybf@ic.unicamp.br}

\author[1]{\fnm{Orlando} \sur{Lee}}\email{lee@ic.unicamp.br}
\equalcont{This author was supported by CNPq Proc. 303766/2018-2, CNPq Proc 425340/2016-3 and FAPESP Proc. 2015/11937-9. ORCID: 0000-0003-4462-3325.}

\affil[1]{\orgdiv{Institute of Computing}, \orgname{State University of Campinas}, \orgaddress{\street{Albert Einstein}, \city{Campinas}, \postcode{13083-852}, \state{São Paulo}, \country{Brazil}}}


\abstract{
Let $D$ be a digraph. A subset $S$ of $V(D)$ is a \emph{stable set} if every pair of vertices in $S$ is non-adjacent in $D$. A collection of disjoint paths $\sP$ of $D$ is a \emph{path partition} of $V(D)$, if every vertex in $V(D)$ is on a path of $\sP$. We say that a stable set $S$ and a path partition $\sP$ are \emph{orthogonal} if each path of $P$ contains exactly one vertex of $S$. A digraph $D$ satisfies the $\alpha$\emph{-property} if for every maximum stable set $S$ of $D$, there exists a path partition $\sP$ such that $S$ and $\sP$ are orthogonal. A digraph $D$ is $\alpha$\emph{-diperfect} if every induced subdigraph of $D$ satisfies the $\alpha$-property. In 1982, Claude Berge proposed a characterization of $\alpha$-diperfect digraphs in terms of forbidden \emph{anti-directed odd cycles}. In 2018, Sambinelli, Silva and Lee proposed a similar conjecture. A digraph $D$ satisfies the \emph{Begin-End-property} or \emph{BE-property} if for every maximum stable set $S$ of $D$, there exists a path partition $\sP$ such that (i)~$S$ and $\sP$ are orthogonal and (ii)~for each path $P\in\sP$, either the start or the end of $P$ lies in $S$. A digraph $D$ is \emph{BE-diperfect} if every induced subdigraph of $D$ satisfies the BE-property. Sambinelli, Silva and Lee proposed a characterization of BE-diperfect digraphs in terms of forbidden \emph{blocking odd cycles}. In this paper, we show some structural results for $\alpha$-diperfect and BE-diperfect digraphs. In particular, we show that in every minimal counterexample $D$ to both conjectures, the size of a maximum stable set is smaller than $\vert V(D)\vert /2$. As an application we use these results to prove both conjectures for arc-locally in-semicomplete and arc-locally out-semicomplete digraphs.}




\maketitle

\section{Notation}
\label{nota}

We consider that the reader is familiar with the basic concepts of graph theory. Thus, this section is mainly concerned with establishing the notation used. For details that are not present in this paper, we refer the reader to Bang-Jensen and Gutin's book~\cite{bang2008digraphs} or to Bondy and Murty's book~\cite{Bondy08}. 

Let $D$ be a digraph with vertex set $V(D)$ and edge set $E(D)$. We only consider finite digraphs without loops and multiple edges. Given two vertices $u$ and $v$ of $V(D)$, we say that $u$ \emph{dominates} $v$, denoted by $u \to v$, if $uv \in E(D)$. We say that $u$ and $v$ are \emph{adjacent} if $u \to v$ or $v \to u$; otherwise we say that $u$ and $v$ are \emph{non-adjacent}. If every pair of distinct vertices of $D$ are adjacent, we say that $D$ is a \emph{semicomplete digraph}. A digraph $H$ is a \emph{subdigraph} of $D$ if $V(H)\subseteq V(D)$ and $E(H) \subseteq E(D)$; moreover, if every edge of $E(D)$ with both vertices in $V(H)$ is in $E(H)$, then we say that $H$ is \emph{induced} by $X = V(H)$, and we write $H = D[X]$. If $uv$ is an edge of $D$, then we say that $u$ and $v$ are \emph{incident} in $uv$. We say that two edges are \emph{adjacent} if they have an incident vertex in common; otherwise we say that they are \emph{non-adjacent}. We say that a digraph $H$ is \emph{inverse} of $D$ if $V(H) = V(D)$ and $E(H)= \{uv : vu \in E(D)\}$.

We say that a vertex $u$ is an \emph{in-neighbor} (resp., \emph{out-neighbor}) of a vertex $v$ if $u \to v$ (resp., $v \to u$). Let $X$ be a subset of $V(D)$. We denote by $N^-(X)$ (resp., $N^+(X)$) the set of vertices in $V(D)-X$ that are in-neighbors (resp., out-neighbors) of some vertex of $X$. We define the \emph{neighborhood} of $X$ as $N(X)=N^-(X) \cup N^+(X)$; when $X=\{v\}$, we write $N^-(v)$, $N^+(v)$ and $N(v)$. We say that $v$ is a \emph{source} if $N^-(v)=\emptyset$ and a \emph{sink} if $N^+(v)=\emptyset$. Furthermore, we define the \emph{neighborhood} of a subset $X$ in a graph $G$, denoted by $N(X)$, as the set of vertices in $V(G)-X$ that are adjacent of some vertex of $X$.

For disjoint subsets $X$ and $Y$ of $V(D)$ (or subdigraphs of $D$), we say that $X$ and $Y$ are \emph{adjacent} if some vertex of $X$ and some vertex of $Y$ are adjacent; $X \to Y$ means that every vertex of $X$ dominates every vertex of $Y$, $X \Rightarrow Y$ means that there exists no edge from $Y$ to $X$ and $X \mapsto Y$ means that both of $X \to Y$ and $X \Rightarrow Y$ hold. When $X = \{x\}$ or $Y = \{y\}$, we write $x \mapsto Y$ and $X \mapsto  y$.

A \emph{path} $P$ in a digraph $D$ is a sequence of distinct vertices $P = v_1v_2 \dots v_k$, such that for all $v_i \in V(P)$, $v_iv_{i+1} \in E(D)$, for $1 \leq i \leq k-1$. We say that $P$ \emph{starts} at $v_1$ and \emph{ends} at $v_k$; to emphasize this fact we may write $P$ as $v_1Pv_k$. We define the \emph{length} of $P$ as $k-1$. We denote by $P_k$ the class of isomorphism of a path of length $k-1$. For disjoint subsets $X$ and $Y$ of $V(D)$ (or subdigraphs of $D$), we say that $X$ \emph{reaches} $Y$ if there are $u \in X$ and $v \in Y$ such that there exists a path from $u$ to $v$ in $D$. The \emph{distance} from $u \in V(D)$ to $v \in V(D)$, denoted by $\textrm{dist}(u,v)$, is the length of the shortest path from $u$ to $v$. The distance from $X$ to $Y$ is $\textrm{dist}(X,Y)= \min\{\textrm{dist}(u,v): u\in X$ and $v \in Y \}$.

A \emph{cycle} $C$ in a digraph $D$ is a sequence of vertices $C = v_1v_2 \dots v_kv_1$ such that $v_1v_2 \dots v_k$ is a path, $v_kv_1 \in E(D)$ and $k>1$. We define the \emph{length} of $C$ as $k$. If $k$ is odd, then we say that $C$ is an \emph{odd cycle}. We say that $D$ is an \emph{acyclic digraph} if $D$ does not contain cycles. The \emph{underlying graph} of a digraph $D$, denoted by $U(D)$, is the simple graph defined by $V(U(D))= V(D)$ and $E(U(D))= \{uv : u $ and $v$ are adjacent in $D\}$. We say that $C$ is a \emph{non-oriented cycle} if $C$ is not a cycle in $D$, but $U(C)$ is a cycle in $U(D)$.

Let $D$ be a digraph. A subset $S$ of $V(D)$ is a \emph{stable set} if every pair of vertices in $S$ is non-adjacent in $D$. The cardinality of a maximum stable set in $D$ is called the \emph{stability number} and is denoted by $\alpha(D)$. A collection of disjoint paths $\sP$ of $D$ is a \emph{path partition} of $V(D)$, if every vertex in $V(D)$ belongs to exactly one path of $\sP$. Let $S$ be a stable set of $D$. We say that $S$ and $\sP$ are \emph{orthogonal} if $\vert V(P) \cap S\vert  = 1$ for every $P \in \sP$.

Let $G$ be a connected graph. A \emph{clique} is a set of pairwise adjacent vertices of $G$. The \emph{clique number} of $G$, denoted by $\omega(G)$, is the size of maximum clique of $G$. We say that a vertex set $B \subset V(G)$ is a \emph{vertex cut} if $G-B$ is a disconnected graph. If $G[B]$ is a complete graph, then we say that $B$ is a \emph{clique cut}. A \emph{(proper) coloring} of $G$ is a partition of $V(G)$ into stable sets $\{S_1,\ldots, S_k\}$. The \emph{chromatic number} of $G$, denoted by $\chi(G)$, is the cardinality of a minimum coloring of $G$. We say that $G$ is \emph{perfect} if for every induced subgraph $H$ of $G$, the equality $\omega(H)=\chi(H)$ holds. We say that a digraph $D$ is \emph{diperfect} if $U(D)$ is perfect.

A \emph{matching} $M$ in a graph $G$ is a set of pairwise non-adjacent edges of $G$. We denoted by $V(M)$ the set of vertices incident on the edges of $M$. We say that a vertex $v$ is \emph{covered} by $M$ if $v \in V(M)$. We also say that $M$ is a matching covering $X \subseteq V(G)$ if $X \subseteq V(M)$. An \emph{$M$-alternating path} $P$ in $G$ is a path whose edges are alternately in $M$ and $E(G)-M$. If neither the start nor the end of $P$ is covered by $M$, then $P$ is called an \emph{$M$-augmenting path}. A matching $M$ in $G$ is \emph{perfect} if it covers $V(G)$. We say that a subset of edges of a digraph $D$ is a \emph{matching} if its corresponding set of edges in $U(D)$ is a matching. Moreover, we denote a bipartite (di)graph $G$ with bipartition $(X,Y)$ by $G[X,Y]$.

\section{Introduction}
\label{intro}

Some very important results in graph theory characterize a certain class of graphs (or digraphs) in terms of certain forbidden induced subgraphs (subdigraphs). The most famous one is probably Berge's Strong Perfect Graph Conjecture~\cite{chudnovsky2006strong}. Berge showed that neither an odd cycle of length at least five nor its complement is perfect. He conjectured that a graph $G$ is perfect if and only if it contains neither an odd cycle of length at least five nor its complement as an induced subdigraph. In 2006, Chudnovsky, Robertson, Seymour and Thomas~\cite{chudnovsky2006strong} proved Berge's conjecture, which became known as the Strong Perfect Graph Theorem.

\begin{theorem}[Chudnovsky, Robertson, Seymour and Thomas, 2006]
\label{perf}
A graph $G$ is perfect if and only if $G$ contains neither an odd cycle of length at least five nor its complement as an induced subgraph. 
\end{theorem}

In this paper we are concerned with two conjectures on digraphs which are somehow similar to Berge's conjecture. Those conjectures relate \emph{path partitions} and \emph{stable sets}. We need a few definitions in order to present both conjectures. 

Let $S$ be a stable set of a digraph $D$. An \emph{$S$-path partition} of $D$ is a path partition $\sP$ such that $S$ and $\sP$ are orthogonal. We say that $D$ satisfies the \emph{$\alpha$-property} if for every maximum stable set $S$ of $D$ there exists an $S$-path partition of $D$, and we say that $D$ is \emph{$\alpha$-diperfect} if every induced subdigraph of $D$ satisfies the $\alpha$-property. A digraph $C$ is an \emph{anti-directed odd cycle} if $({\rm i})$ $U(C) = x_1x_2 \dots x_{2k+1}x_1$ is a cycle, where $k \geq 2$ and $(\rm{ii})$ each of the vertices $x_1,x_2,x_3,x_4,x_6,x_8, \ldots, x_{2k}$ is either a source or a sink (see Figure~\ref{circ-berge}).

\begin{figure}[htbp]
\center
\subfloat[]{
    \tikzset{middlearrow/.style={
	decoration={markings,
		mark= at position 0.6 with {\arrow{#1}},
	},
	postaction={decorate}
}}

\tikzset{shortdigon/.style={
	decoration={markings,
		mark= at position 0.45 with {\arrow[inner sep=10pt]{<}},
		mark= at position 0.75 with {\arrow[inner sep=10pt]{>}},
	},
	postaction={decorate}
}}

\tikzset{digon/.style={
	decoration={markings,
		mark= at position 0.4 with {\arrow[inner sep=10pt]{<}},
		mark= at position 0.6 with {\arrow[inner sep=10pt]{>}},
	},
	postaction={decorate}
}}

\begin{tikzpicture}[scale = 0.5]		
	\node (n4) [black vertex] at (6.5,10) {};
	\node (n2) [black vertex] at (9,5)  {};
	\node (n3) [black vertex] at (9,8)  {};
	\node (n5) [black vertex] at (4,8)  {};
	\node (n1) [black vertex] at (4,5)  {};
	
	\node (label_n4)  at (6.5,10.5) {$v_4$};
	\node (label_n2)  at (9.5,4.5)  {$v_2$};
	\node (label_n3)  at (9.5,8.5)  {$v_3$};
	\node (label_n5)  at (3.5,8.5)  {$v_5$};
	\node (label_n1)  at (3.5,4.5)  {$v_1$};

  \foreach \from/\to in {n1/n2,n3/n2,n3/n4,n5/n4,n1/n5}
    \draw[edge,middlearrow={>}] (\from) -- (\to);    
\end{tikzpicture}
}
\quad
\subfloat[]{
    \tikzset{middlearrow/.style={
	decoration={markings,
		mark= at position 0.6 with {\arrow{#1}},
	},
	postaction={decorate}
}}

\tikzset{shortdigon/.style={
	decoration={markings,
		mark= at position 0.45 with {\arrow[inner sep=10pt]{<}},
		mark= at position 0.75 with {\arrow[inner sep=10pt]{>}},
	},
	postaction={decorate}
}}

\tikzset{digon/.style={
	decoration={markings,
		mark= at position 0.4 with {\arrow[inner sep=10pt]{<}},
		mark= at position 0.6 with {\arrow[inner sep=10pt]{>}},
	},
	postaction={decorate}
}}

\begin{tikzpicture}[scale = 0.5]		

	\node (n1) [black vertex] at (4,5)  {};
	\node (n2) [black vertex] at (9,5)  {};
	\node (n3) [black vertex] at (9,7)  {};
    \node (n4) [black vertex] at (9,9)  {};
	\node (n5) [black vertex] at (6.5,11) {};
	\node (n6) [black vertex] at (4,9)  {};
	\node (n7) [black vertex] at (4,7)  {};

	\node (label_n1)  at (3.5,4.5)  {$v_1$};
	\node (label_n2)  at (9.5,4.5)  {$v_2$};
	\node (label_n3)  at (9.5,7.5)  {$v_3$};
	\node (label_n4)  at (9.5,9.5)  {$v_4$};	
	\node (label_n5)  at (6.5,11.5) {$v_5$};
	\node (label_n6)  at (3.5,9.5)  {$v_6$};
	\node (label_n7)  at (3.5,7.5)  {$v_7$};

  \foreach \from/\to in {n1/n2,n3/n2,n1/n7,n6/n7,n6/n5,n5/n4,n3/n4}
    \draw[edge,middlearrow={>}] (\from) -- (\to);    
\end{tikzpicture}
}
\caption{\centering Examples of anti-directed odd cycles with length five and seven, respectively.}
\label{circ-berge}
\end{figure}
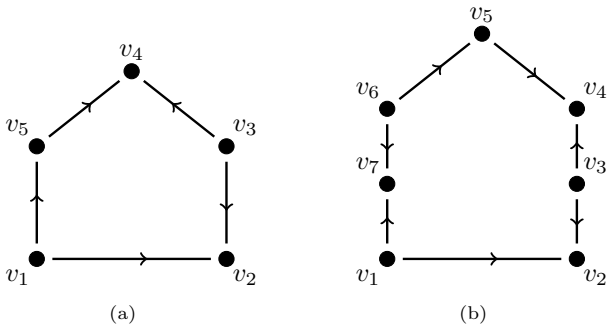

Berge~\cite{berge1981} showed that anti-directed odd cycles do not satisfy the $\alpha$-property, and hence, they are not $\alpha$-diperfect, which led him to conjecture the following characterization for $\alpha$-diperfect digraphs.

\begin{conjecture}[Berge, 1982]
\label{conj_berge}
A digraph $D$ is $\alpha$-diperfect if and only if $D$ does not contain an anti-directed odd cycle as an induced subdigraph. 
\end{conjecture}

Denote by $\sB$ the set of all digraphs which do not contain an induced anti-directed odd cycle. So Berge's conjecture can be stated as: $D$ is $\alpha$-diperfect if and only if $D$ belongs to $\sB$. In 1982, Berge~\cite{berge1981} verified Conjecture~\ref{conj_berge} for diperfect digraphs and for symmetric digraphs (digraphs such that if $uv \in E(D)$, then $vu \in E(D)$). In the next three decades, no results regarding this problem were published. In 2018, Sambinelli, Silva and Lee~\cite{tesemaycon2018,ssl} verified Conjecture~\ref{conj_berge} for locally in-semicomplete digraphs and digraphs whose underlying graph is series-parallel. To the best of our knowledge these are the only particular cases verified for this conjecture. For ease of reference, we state the following result.

\begin{lemma}[Berge, 1982]
\label{diper-alpha}
Let $D$ be a diperfect digraph. Then, $D$ is $\alpha$-diperfect.
\end{lemma}

In an attempt to understand the main difficulties in proving Conjecture~\ref{conj_berge}, Sambinelli, Silva and Lee~\cite{tesemaycon2018,ssl} introduced the class of Begin-End-diperfect digraphs, or simply BE-diperfect digraphs, which we define next.

Let $S$ be a stable set of a digraph $D$. A path partition $\sP$ is an \emph{$S_{BE}$-path partition} of $D$ if $(\rm{i})$ $\sP$ and $S$ are orthogonal and $(\rm{ii})$ every vertex of $S$ starts or ends a path at $\sP$. We say that $D$ satisfies the \emph{BE-property} if for every maximum stable set of $D$ there exists an $S_{BE}$-path partition, and we say that $D$ is \emph{BE-diperfect} if every induced subdigraph of $D$ satisfies the BE-property. Note that if $D$ is BE-diperfect, then it is also $\alpha$-diperfect, but the converse is not true (see the digraph in Figure~\ref{circ-bloqueante}(b)). A digraph $C$ is a \emph{blocking odd cycle} if $(\rm{i})$ $U(C) = x_1 x_2 \dots x_{2k+1}x_1$ is a cycle, where $k \geq 1$ and $(\rm{ii})$ $x_1$ is a source and $x_2$ is a sink (see Figure~\ref{circ-bloqueante}). Note that every anti-directed odd cycle is also a blocking odd cycle. In the special case $k=1$, we say that $D$ is a \emph{transitive triangle} (see Figure~\ref{circ-bloqueante}(b)).

\begin{figure}[htbp]
\center
\subfloat[]{
	    \tikzset{middlearrow/.style={
	decoration={markings,
		mark= at position 0.6 with {\arrow{#1}},
	},
	postaction={decorate}
}}

\tikzset{shortdigon/.style={
	decoration={markings,
		mark= at position 0.45 with {\arrow[inner sep=10pt]{<}},
		mark= at position 0.75 with {\arrow[inner sep=10pt]{>}},
	},
	postaction={decorate}
}}

\tikzset{digon/.style={
	decoration={markings,
		mark= at position 0.4 with {\arrow[inner sep=10pt]{<}},
		mark= at position 0.6 with {\arrow[inner sep=10pt]{>}},
	},
	postaction={decorate}
}}

\begin{tikzpicture}[scale = 0.5,auto=left]
	\node (n4) [black vertex] at (6.5,10) {};
	\node (n2) [black vertex] at (9,5)  {};
	\node (n3) [black vertex] at (9,8)  {};
	\node (n5) [black vertex] at (4,8)  {};
	\node (n1) [black vertex] at (4,5)  {};
	
	\node (label_n4)  at (6.5,10.5) {$v_4$};
	\node (label_n2)  at (9.5,4.5)  {$v_2$};
	\node (label_n3)  at (9.5,8.5)  {$v_3$};
	\node (label_n5)  at (3.5,8.5)  {$v_5$};
	\node (label_n1)  at (3.5,4.5)  {$v_1$};

  \foreach \from/\to in {n1/n2,n3/n2,n3/n4,n1/n5}
    \draw[edge,middlearrow={>}] (\from) -- (\to);
    
  \foreach \from/\to in {n5/n4}
    \draw[edge,digon] (\from) -- (\to);      
\end{tikzpicture}
}
\quad
\subfloat[]{
    \input{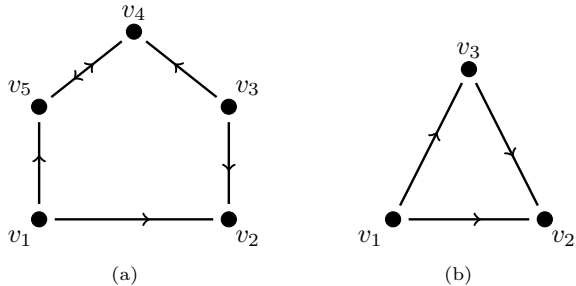}
}
\caption{\centering Examples of blocking odd cycles with length five and three, respectively. We also say that the digraph in (b) is a transitive triangle.}
\label{circ-bloqueante}
\end{figure}

Sambinelli, Silva and Lee~\cite{tesemaycon2018,ssl} showed that blocking odd cycles do not satisfy the BE-property, and hence, they are not BE-diperfect, which led them to conjecture the following characterization of BE-diperfect digraphs.

\begin{conjecture}[Sambinelli, Silva and Lee, 2018]
\label{conj_be}
A digraph $D$ is BE-diperfect if and only if $D$ does not contain a blocking odd cycle as an induced subdigraph.
\end{conjecture}

Denote by $\D$ the set of all digraphs which do not contain an induced blocking odd cycle. So Conjecture~\ref{conj_be} can be stated as: $D$ is BE-diperfect if and only if $D$ belongs to $\D$. Sambinelli, Silva and Lee~\cite{tesemaycon2018,ssl} verified Conjecture~\ref{conj_be} for locally in-semicomplete digraphs and digraphs whose underlying graph are series-parallel or perfect. To the best of our knowledge these are the only particular cases verified for this conjecture. Note that a diperfect digraph belongs to $\D$ if and only if it contains no induced transitive triangle. For ease of reference, we state the following result.

\begin{lemma}[Sambinelli, Silva and Lee, 2018]
\label{diperD-BE}
Let $D$ be a diperfect digraph. If $D\in\D$, then $D$ is BE-diperfect.
\end{lemma}

The rest of this paper is organized as follows. In Section~\ref{struc-results}, we present some structural results which may be useful on approaching the general conjectures. In particular, we show that if a digraph $D$ is a minimal counterexample for Conjecture~\ref{conj_berge} or Conjecture~\ref{conj_be}, then $\alpha(D) < \frac{\vert V(D)\vert }{2}$. In Section~\ref{arc-semi}, we provide some results on structure of an arc-locally in-semicomplete digraph. In Section~\ref{be-alis}, we verify Conjecture~\ref{conj_be} for arc-locally in-semicomplete and for arc-locally out-semicomplete digraphs. In Section~\ref{berge-alis}, we verify Conjecture~\ref{conj_berge} for arc-locally in-semicomplete and for arc-locally out-semicomplete digraphs. Finally, in Section~\ref{conclu}, we present some conclusions. 

\section{Some structural results}
\label{struc-results}

In this section, we present some structural results for BE-diperfect digraphs and $\alpha$-diperfect digraphs. For the three initial lemmas, we need the celebrated Hall's theorem~\cite{hall1935} and Berge's theorem~\cite{berge1957} about matching. 

\begin{theorem}[Hall, 1935]
\label{teo-hall}
A bipartite graph $G := G[X,Y]$ has a matching covering $X$ if and only if $\vert N(W)\vert  \geq \vert W\vert $ for all $W \subseteq X$.
\end{theorem}

\begin{theorem}[Berge, 1957]
\label{berge-teo}
A matching $M$ in a graph $G$ is a maximum matching if and only if $G$ has no $M$-augmenting path.
\end{theorem}

\begin{lemma}
\label{arc-in-hall}
Let $ G:=G[X,Y]$ be a bipartite graph. If $G$ has no matching covering $X$, then there exists a non-empty subset $X' \subseteq X$ such that $G[X' \cup N(X')]$ has a matching covering $N(X')$.
\end{lemma}

\begin{proof}
Assume that there exists no matching covering $X$ in $G$. By Theorem~\ref{teo-hall}, there exists a subset $W$ of $X$ such that $ \vert N(W)\vert <\vert W\vert $; choose such $W$ as small as possible. By the choice of $W$, for every $X' \subset W$ (and hence, for $X' \subset X$), it follows that $\vert N(X')\vert  \geq \vert X'\vert $. Let $X'$ be a subset of $W$ with the same size as $\vert N(W)\vert $. Since for every $X^* \subseteq X'$, it follows that $\vert N(X^*)\vert  \geq \vert X^*\vert $, we conclude by Theorem~\ref{teo-hall} that the graph $G[X' \cup N(X')]$ has a matching covering $X'$ (and hence, $N(W)$).
\end{proof}

\begin{lemma}
\label{arc-in-matching-N(S)}
Let $S$ be a maximum stable set in a digraph $D$. Let $X$ be a stable set disjoint from $S$ and let $Y = N(X) \cap S$. Then, there exists a matching between $X$ and $Y$ covering $X$.
\end{lemma}

\begin{proof}
Towards a contradiction, assume that there exists no matching between $X$ and $Y$ covering $X$. By Theorem~\ref{teo-hall}, there exists a subset $W$ of $X$ such that $\vert N(W)\vert <\vert W\vert $. Since $X \cap S = \emptyset$, it follows that $(S-N(W)) \cup W$ is a stable set larger than $S$ in $D$, a contradiction.
\end{proof}

\begin{lemma}
\label{hall-x-y-emp}
Let $G:=G[X,Y]$ be a bipartite graph which has a matching covering $X$. Then, for every $Y' \subset Y$, there exists a matching $M$ covering $X$ such that the restriction of $M$ to $G[X' \cup Y']$, where $X'=N(Y')$, is a maximum matching of $G[X' \cup Y']$. 
\end{lemma}

\begin{proof}
Let $Y' \subset Y$. Let $H:=H[X',Y']$ be a bipartite subgraph corresponding to $G[X' \cup Y']$, where $X'=N(Y')$. Let $M$ be a matching covering $X$ such that $\vert M\cap E(H)\vert $ as maximum as possible. Let $M'= M \cap E(H)$. Towards a contradiction, assume that $M'$ is not maximum in $H$. By Theorem~\ref{berge-teo}, there exists an $M'$-augmenting path $uPv$ in $H$. Since $P$ is odd, we may assume that $u \in Y'$ and $v \in X'$. Since $M$ covers $X$, there exists $w \in Y-Y'$ such that $wv \in M$. Since $X'=N(Y')$, $u$ is not covered by an edge of $M$. Thus $M^* = ((M - E(P)) \cup (E(P)-M)) - wv$ is a maximum matching covering $X$ such that $\vert M^* \cap E(H)\vert  > \vert M \cap E(H)\vert $, a contradiction.  
\end{proof}

The next two lemmas are important tools that we use in the forthcoming sections.

\begin{lemma}
\label{arc-in-n-cobre-S}
Let $D$ be a digraph such that every proper induced subdigraph of $D$ satisfies the BE-property. Let $S$ be a maximum stable set of $D$. If there exists no matching between $S$ and $ N(S)$ covering $S$, then $D$ has an $S_{BE}$-path partition.	
\end{lemma}

\begin{proof}
Let $H$ be the bipartite digraph obtained from $D[S \cup N(S)]$ by removing all edges connecting vertices in $N(S)$. Since there exists no matching covering $S$ in $H$, by Lemma~\ref{arc-in-hall} there exists $X \subset S$ such that $H[X \cup N(X)]$ has a matching $M$ covering $N(X)$. Let $D'=D-N(X)$. Since $N(X) \cap S = \emptyset$, $S$ is a maximum stable set in $D'$. By hypothesis, $D'$ is BE-diperfect. Let $\sP'$ be a $ S_{BE}$-path partition of $D'$. Since $V(D')\cap N(X)=\emptyset$, every vertex in $X$ is a path in $\sP'$. Let $\sP_M$ be the set of paths in $D$ corresponding to the edges in $M$. Thus the collection $(\sP' - (X \cap V(M))) \cup \sP_M$ is an $S_{BE}$-path partition of $D$. 
\end{proof}

Since an $S_{BE}$-path partition is also an $S$-path partition, we conclude the following result.

\begin{lemma}
\label{arc-in-n-cobre-S-berge}
Let $D$ be a digraph such that every proper induced subdigraph of $D$ satisfies the $\alpha$-property and let $S$ be a maximum stable set of $D$. If there is no matching between $S$ and $ N(S)$ covering $S$, then $D$ has an $S$-path partition.	
\qed
\end{lemma}

The next lemma is very important and will be used extensively throughout this paper.
  
\begin{lemma}
\label{stable_set_S_menor_vizinhanca}
Let $D$ be a digraph such that every proper induced subdigraph of $D$ satisfies the BE-property. If $D$ has a stable set $Z$ such that $\vert N(Z)\vert  \leq \vert Z\vert $, then $D$ satisfies the BE-property.
\end{lemma}

\begin{proof}
Let $S$ be a maximum stable set of $D$. Since $S$ is arbitrary, to show that $D$ satisfies the BE-property it suffices to show that $D$ has an $S_{BE}$-path partition. First, we prove that there exists a perfect matching between $Z$ and $N(Z)$. Let $Y=N(Z)$. Since $S$ is maximum, then $\vert Z-S\vert  \leq \vert Y \cap S\vert $. Since $\vert Z\vert  \geq \vert Y\vert $, this implies that $\vert Z \cap S\vert  \geq \vert Y-S\vert $. By Lemma~\ref{arc-in-n-cobre-S}, we may assume that there exists a matching $M_1$ between $Z \cap S$ and $Y-S$ covering $Z \cap S$. Since $\vert Z\vert  \geq \vert Y\vert $ and $\vert Z-S\vert  \leq \vert Y \cap S\vert $, it follows that $\vert Z \cap S\vert =\vert Y-S\vert $ and $\vert Z-S\vert =\vert Y \cap S\vert $. By Lemma~\ref{arc-in-matching-N(S)}, there exists a matching $M_2$ between $Z-S$ and $Y \cap S$ covering $Z-S$. Thus, the matching $M = M_1 \cup M_2$ is a perfect matching between $Z$ and $Y$. Let $\sP_M$ be the set of paths in $D$ corresponding to the edges of $M$. Note that $\sP_M$ and $S$ are orthogonal. Let $S' = S - V(M)$ and let $D' = D- V(M)$. Let $k = \vert S \cap V(M)\vert  = \vert Z\vert $ and note that $\vert S'\vert =\vert S\vert -k$. Assume that $S'$ is not a maximum stable set of $D'$. Let $S^*$ be a maximum stable set of $D'$. Since $\vert S^*\vert  > \vert S\vert -k$ and $V(D') \cap (Z \cup Y)= \emptyset$, it follows that $S^* \cup Z$ is a stable set larger than $S$ in $D$, a contradiction. By hypothesis, $D'$ is BE-diperfect. Let $\sP'$ be an $S'_{BE}$-path partition of $D'$. Thus the collection $\sP' \cup \sP_M$ is an $S_{BE}$-path partition of $D$.
\end{proof}

The next two theorems state that minimal counterexamples to Conjectures \ref{conj_berge} and \ref{conj_be} cannot have large stability number.

\begin{theorem}
\label{arc-in-maior-S}
Let $D$ be a digraph such that every proper induced subdigraph of $D$ satisfies the BE-property. If $\alpha(D) \geq \frac {\vert V(D)\vert }{2}$, then $D$ satisfies the BE-property.
\end{theorem}

\begin{proof}
Let $S$ a maximum stable set of $D$. Let $\overline{S} = V(D)-S$. By hypothesis, it follows that $\vert S\vert  \geq \vert \overline {S}\vert $, and hence, the result follows by Lemma~\ref{stable_set_S_menor_vizinhanca}. 
\end{proof}

We omit the proof from the next theorem, since it is analogous to the proof of Theorem~\ref{arc-in-maior-S}, but we use Lemma~\ref{arc-in-n-cobre-S-berge} instead of Lemma~\ref{arc-in-n-cobre-S}.

\begin{theorem}
\label{arc-in-maior-S-be}
Let $D$ be a digraph such that every proper induced subdigraph of $D$ satisfies the $\alpha$-property. If $\alpha(D) \geq \frac {\vert V(D)\vert }{2}$, then $D$ satisfies the $\alpha$-property. 
\qed
\end{theorem}

The next three lemmas are more specific structural results and are used in Sections~\ref{be-alis} and \ref{berge-alis} to verify Conjectures~\ref{conj_be} and \ref{conj_berge} for arc-locally (out) in-semicomplete digraphs, respectively.

\begin{lemma}
\label{arc-in-par-W-B-U}
Let $D$ be a digraph such that every proper induced subdigraph of $D$ satisfies the BE-property. If $V(D)$ contains disjoint nonempty subsets $U,X,Y$ such that $X$ and $Y$ are stable, $N(Y) \subseteq X$, $N(X) \subseteq U \cup Y$ and every vertex in $U$ is adjacent to every vertex in $X$, then $D$ satisfies the BE-property.
\end{lemma}

\begin{proof}
Let $S$ be a maximum stable set of $D$. To show that $D$ satisfies the BE-property, it suffices to show that $D$ has an $S_{BE}$-path partition. Note that $N(Y \cap S) \subseteq X-S$ and $N(X \cap S) \cap Y \subseteq Y-S$. It follows by Lemma~\ref{arc-in-matching-N(S)} that there exists a matching $M_1$ between $Y-S$ and $X \cap S$ covering $Y-S$. By Lemma~\ref{arc-in-n-cobre-S}, we may assume that there exists a matching $M_2$ between $Y \cap S$ and $X-S$ covering $Y \cap S$. Let $M=M_1 \cup M_2$ be a matching, and note that $M$ covers $Y$. Let $D'= D - V(M)$ and let $S'= S - V(M)$. Let $k = \vert S \cap V(M)\vert =\vert Y\vert = \vert V(M) \cap X\vert $. Assume that $S'$ is not a maximum stable in $D'$. Let $Z$ be a maximum stable set of $D'$. Thus, $\vert Z\vert  > \vert S'\vert  = \vert S\vert -k$. If $U \cap Z \neq \emptyset$, then since every vertex in $U$ is adjacent to every vertex in $X$, it follows that $X \cap Z= \emptyset$. Since $N(Y) \cap U = \emptyset$ and $Y$ is stable, it follows that the set $Z \cup Y$ is stable and larger than $S$ in $D$, a contradiction. So we may assume that $U \cap Z = \emptyset$. Since $Y \cap Z=\emptyset$, $N(X)\subseteq U \cup Y$ and $X$ is stable, it follows that $Z \cup (V(M) \cap X)$ is a stable set larger than $S $ in $D$, a contradiction. Therefore, $S'$ is a maximum stable in $D'$. By hypothesis, $D'$ is BE-diperfect. Let $\sP'$ be an $S'_{BE}$-path partition of $D'$. Let $\sP_M$ be the set of paths in $D$ corresponding to the edges of $M$. Note that $\sP_M$ and $S$ are orthogonal. Thus the collection $\sP' \cup \sP_M$ is an $S_{BE}$-path partition of $D$.
\end{proof}

\begin{lemma}
\label{arc-in-H-W-U}
Let $D$ be a digraph such that every proper induced subdigraph of $D$ satisfies the BE-property. Let $S$ be a maximum stable set of $D$. If $D$ contains a connected induced bipartite subdigraph $H:=H[X,Y]$ such that $Y \subseteq S$, $N(X) \cap S = Y$, $N(X) \cap N(Y)=\emptyset$ and every vertex in $N(Y)-X$ is adjacent to every vertex in $N(X)$, then $D$ admits an $S_{BE}$-path partition.  
\end{lemma}

\begin{proof}

Note that $X \cap S = \emptyset$ because $H$ is connected and $Y \subseteq S$. Since $S$ is a maximum stable set and $N(X) \cap S = Y$, it follows by Lemma~\ref{arc-in-matching-N(S)} that there exists a matching $M$ between $X$ and $Y$ covering $X$. Let $D'=D-V(M)$ and let $S'=S-V(M)$. Note that $V(D') \cap X=\emptyset$. Towards a contradiction, assume that $S'$ is not a maximum stable set in $D'$ and let $Z$ be a maximum stable set in $D'$. Note that $\vert Z\vert  > \vert S\vert -\vert V(M)\cap Y\vert $ and $\vert V(M)\cap Y\vert  = \vert X\vert $. If $Z \cap (N(Y)-X) = \emptyset$, then $V(D') \cap X = \emptyset$ and we conclude that $Z \cup (V(M) \cap Y)$ is a stable set in $D$ larger than $S$, a contradiction. So we may assume that $Z \cap (N(Y)-X)  \neq \emptyset$. Since every vertex in $N(Y)-X$ is adjacent to every vertex in $N(X)$ and $N(X) \cap N(Y)=\emptyset$, it follows that $N(X) \cap Z = \emptyset$. Thus $Z \cup X$ is a stable set in $D$ larger than $S$ in $D$, a contradiction. Therefore, $S'$ is a maximum stable set in $D'$. Let $\sP_M$ be the collection of paths corresponding to the edges of $M$. By hypothesis, $D'$ is BE-diperfect. Let $\sP'$ be an $S'_{BE}$-path partition of $D'$. Thus the collection $\sP' \cup \sP_M$ is an $S_{BE}$-path partition of $D$.
\end{proof}

\begin{lemma}
\label{bip-x-y-font}
Let $D$ be a digraph such that every proper induced subdigraph of $D$ satisfies the BE-property. Let $S$ be a maximum stable set of $D$. Let $H:=H[X,Y]$ be an induced bipartite subdigraph of $D$ such that $N^-(X)=Y$, $Y \Rightarrow X$, $Y \cap S=\emptyset$ and $N^+(X \cap S)=\emptyset$. If there exists no matching between $X$ and $Y$ covering $X$, then $D$ has an $S_{BE}$-path partition.	
\end{lemma}

\begin{proof}
By Lemma~\ref{arc-in-hall}, there exists a non-empty subset $X' \subseteq X$ such that $D[X'\cup N^-(X')]$ has a matching $M$ covering $N^-(X')$. Let $Y'=N^-(X')$ and let $D'=D-Y'$.  Since $Y \cap S = \emptyset$, $S$ is a maximum stable set in $D'$. By hypothesis, $D'$ is BE-diperfect. Let $\sP'$ be an $S_{BE}$-path partition of $D$. Since $N^-(X')=Y'$ and $N^+(X' \cap S)=\emptyset$, it follows that every vertex $v$ in $X'$ starts a path in $\sP'$ (if $v \notin S$) or $v$ is itself a path in $\sP'$ (if $v \in S$). Since $Y' \Rightarrow X'$, it is easy to see that using the edges in $M$, we can add the vertices of $Y'$ to paths in $\sP'$ that starts at some vertex in $V(M)\cap X'$, obtaining an $S_{BE}$-path partition of $D$.   
\end{proof}

\section{Arc-locally (out) in-semicomplete digraphs}
\label{arc-semi}

In this section, we extend the results in \cite{freitas2021} and we provide more on structure of an arc-locally in-semicomplete digraph. Let $D$ be a digraph. We say that $D$ is \emph{arc-locally in-semicomplete} (resp., \emph{arc-locally out-semicomplete}) if for each edge $uv \in E(D)$, every in-neighbor (resp., out-neighbor) of $u$ and every in-neighbor (resp., out-neighbor) of $v$ are adjacent or are the same vertex. Note that the inverse of an arc-locally in-semicomplete digraph is an arc-locally out-semicomplete digraph.

Arc-locally (out) in-semicomplete digraphs were introduced by Bang-Jensen~\cite{alis} as a common generalization of semicomplete and semicomplete bipartite digraphs. Since then, these classes have been extensively studied in the literature~\cite{wang2009structure,wang2011,wang2019critical, galeana2009,galeana2012,freitas2021}.

Let us start with a class of digraphs which are closely related to arc-locally in-semicomplete digraphs. Let $Q$ be a cycle of length $k \geq 2$ and let $X_1 , X_2 , \ldots , X_k$ be disjoint stable sets. The \emph{extended cycle} $Q:= Q[ X_1 , X_2 , \ldots , X_k]$ is the digraph with vertex set $X_1 \cup X_2 \cup \cdots \cup X_k$ and edge set $\{x_ix_{i+1} : x_i \in X_i, x_{i+1} \in X_{i + 1}, i=1,2,\ldots,k \}$, where subscripts are taken modulo $k$. So $X_1 \mapsto X_2 \mapsto \cdots \mapsto X_k \mapsto X_1$. An extended cycle is \emph{odd} if $k$ is odd (see Figure~\ref{circ-estentido}). 

\begin{figure}[ht]
\centering
    \tikzset{middlearrow/.style={
	decoration={markings,
		mark= at position 0.6 with {\arrow{#1}},
	},
	postaction={decorate}
}}

\tikzset{shortdigon/.style={
	decoration={markings,
		mark= at position 0.45 with {\arrow[inner sep=10pt]{<}},
		mark= at position 0.75 with {\arrow[inner sep=10pt]{>}},
	},
	postaction={decorate}
}}

\tikzset{digon/.style={
	decoration={markings,
		mark= at position 0.4 with {\arrow[inner sep=10pt]{<}},
		mark= at position 0.6 with {\arrow[inner sep=10pt]{>}},
	},
	postaction={decorate}
}}

\def \n {5}
\def \radius {3cm}
\def \margin {8}

\begin{tikzpicture}[scale = 0.66]

    \draw (4,5) circle (30pt);
	\node (v1_1) [black vertex] at (3.8,4.8)  {};
	\node (v1_2) [black vertex] at (4.2,5.3)  {};

	\draw (9,5) circle (30pt); 
	\node (v2_1) [black vertex] at (9,5)  {};

	\draw (9,8) circle (30pt);
	\node (v3_1) [black vertex] at (9.5,8)  {};
	\node (v3_2) [black vertex] at (8.5,8)  {};
	\node (v3_3) [black vertex] at (9,8)  {};

	\draw (6.5,10) circle (30pt);
	\node (v4_1) [black vertex] at (6.5,10.5) {};
	\node (v4_2) [black vertex] at (6.5,9.5) {};

    \draw (4,8) circle (30pt);	
	\node (v5_1) [black vertex] at (4,8)  {};

	\node (label_n4)  at (6.5,11.5) {$X_4$};
	\node (label_n2)  at (10.,4.0)  {$X_2$};
	\node (label_n3)  at (10.0,9.0)  {$X_3$};
	\node (label_n5)  at (3.0,9.0)  {$X_5$};
	\node (label_n1)  at (3.0,4.0)  {$X_1$};
	

  \foreach \from/\to in {v1_1/v2_1,v1_2/v2_1,v2_1/v3_1,v2_1/v3_2,v2_1/v3_3,v3_1/v4_1,v3_2/v4_1,v3_3/v4_1,v3_1/v4_2,v3_2/v4_2,v3_3/v4_2,v4_1/v5_1,v4_2/v5_1,v5_1/v1_1,v5_1/v1_2}
    \draw[edge,middlearrow={>}] (\from) -- (\to);    
\end{tikzpicture}
\caption{\centering Example of an odd extended cycle.}
\label{circ-estentido}
\end{figure}
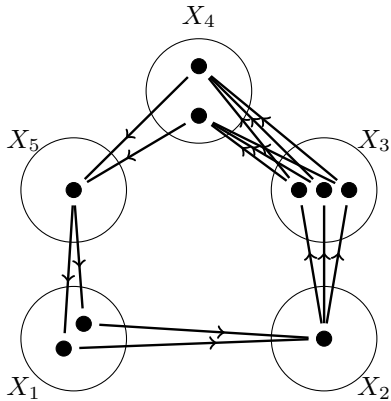

In~\cite{wang2009structure}, Wang and Wang characterized strong arc-locally in-semicomplete digraphs. Recently, Freitas and Lee~\cite{freitas2021} characterized the structure of arbitrary connected arc-locally in-semicomplete digraphs. 

\begin{theorem}[Freitas and Lee, 2021]
\label{lem-arc-resultado}
Let $D$ be a connected arc-locally in-semicomplete digraph. Then,
\begin{enumerate}
	\item[\rm{(i)}] $D$ is a diperfect digraph, or

	\item[\rm{(ii)}] $V(D)$ can be partitioned into $(V_1, V_2, V_3)$ such that $D[V_1]$ is a semicomplete digraph, $V_1 \mapsto V_2$, $V_1 \Rightarrow V_3$, $D[V_2]$ is an odd extended cycle of length at least five, $V_2 \Rightarrow V_3$, $D[V_3]$ is a bipartite digraph and $V_1$ or $V_3$ (or both) can be empty, or 
	
	\item[\rm{(iii)}] $D$ has a clique cut. \newline
\end{enumerate} 
\end{theorem}

The next lemma states that if $V(D)$ admits a partition as described in Theorem~\ref{lem-arc-resultado}\rm{(ii)} and $V_1 = \emptyset$, then $D$ does not contain cycle of length three.

\begin{lemma}
\label{lema_arc-livre-triangulo}
Let $D$ be an arc-locally in-semicomplete digraph. Let $(V_1,V_2,V_3)$ be a partition of $V(D)$ as described in Theorem~\ref{lem-arc-resultado}\rm{(ii)}. Then, the graph $U(D[V_2 \cup V_3])$ does not contain a cycle of length three.
\end{lemma}

\begin{proof}

Let $Q:=Q[X_1,X_2,...,X_k]$ be the odd extended cycle of length at least five corresponding to $D[V_2]$. Since $U(D[V_3])$ is bipartite and $Q$ is an extended cycle of length at least five, it follows that both $U(D[V_3])$ and $U(Q)$ do not contain a cycle of length three. Assume that $U(D[V_2 \cup V_3])$ contains a cycle $T$ of length three. Note that $V(T) \cap V_2 \neq \emptyset$ and $V(T) \cap V_3 \neq \emptyset$. Since $V_2 \Rightarrow V_3$, it follows that $T$ is a transitive triangle in $D[V_2 \cup V_3]$. Let $V(T)=\{x_1,x_2,x_3\}$. Now we consider two cases, depending on the cardinality of $\vert V(T) \cap V_2\vert $. \newline

\textbf{Case 1.} $\vert V(T) \cap V_2\vert =2$. Let $x_1,x_2 \in V_2$ and let $x_3 \in V_3$. Without loss of generality, assume that $x_1x_2 \in E(D)$, $x_1 \in X_1$, and $x_2 \in X_2$. Let $x_k \in X_k$ such that $x_k \to x_1$. Since $x_2 \to x_3$, $x_1x_3 \in E(D)$ and $D$ is arc-locally in-semicomplete, it follows that $x_k$ and $x_2$ are adjacent, a contradiction to the fact that $D[V_2]$ is an odd extended cycle of length at least five. \newline

\textbf{Case 2.} $\vert V(T) \cap V_2\vert =1$. Let $x_1 \in V_2$ and let $x_2,x_3 \in V_3$. Without loss of generality, assume that $x_1 \in X_1$ and $x_2x_3 \in E(D)$. Let $x_k \in X_k$ such that $x_k \to x_1$. Since $x_2 \to x_3$, $x_1x_3 \in E(D)$, $V_2 \Rightarrow V_3$ and $D$ is arc-locally in-semicomplete, it follows that $x_k \to x_2$. Therefore, $D[\{x_k,x_1,x_2\}]$ is a transitive triangle and the result follows by Case 1.
\end{proof}

The next lemma is more specific structural result.

\begin{lemma}
\label{lem_arc_1}
Let $D$ be an arc-locally in-semicomplete digraph. Let $H:=H[X,Y]$ be an induced connected bipartite subdigraph of $D$ such that $\vert X\vert  \geq 1$, $\vert Y\vert  \geq 1$ and $X \Rightarrow Y$. Let $v$ be a vertex of $D-V(H)$ that dominates some vertex of $X$. If $v \Rightarrow X$, then $v \mapsto X$.
\end{lemma}

\begin{proof}
Let $u$ be a vertex in $X$ such that $v \to u$. Let $w$ be a vertex in $X$. Since $H$ is connected, $U(H)$ has a path $P = x_1y_1x_2y_2 \ldots x_{k-1}y_{k-1}x_k$ where $x_1=u$ and $x_k=w$; note that $x_i \in X$ and $y_i \in Y$. We prove by induction that $v$ dominates each vertex $x_i$ in $P$. The base is trivial since $v$ dominates $x_1=u$. Suppose that $v$ dominates $x_{i-1}$. Since $X \Rightarrow Y$, we conclude that $x_{i-1}y_{i-1} \in E(D)$ and $x_i$ dominates $y_{i-1}$. Since $D$ is arc-locally in-semicomplete, $v$ and $x_i$ are adjacent; but $v \Rightarrow X$, and hence, $v \to x_i$. So we conclude that $v$ dominates $w$ and thus $v \mapsto X$.
\end{proof}

For the next lemma, we need to define some sets. Let $D$ be an arc-locally in-semicomplete digraph. Let $(V_1,V_2,V_3)$ be a partition of $V(D)$ as described in Theorem~\ref{lem-arc-resultado}\rm{(ii)}. Recall that $V_1 \mapsto V_2$, $V_1 \cup V_2 \Rightarrow V_3$ and $D[V_2]$ is an odd extended cycle of length at least five. Let $Q:=Q[X_1,X_2,\ldots, X_k]$ be the odd extended cycle corresponding to $D[V_2]$. Let $N_0=V_2$ and for $d \geq 1$ denote by $N_d$ the set of vertices that are at distance $d$ from $V_2$. Note that $N_d \subseteq V_3$ for $d\geq 1$ because $V_1 \mapsto V_2$. For all $i \in \{1,2,\ldots,k\}$, denote by $R_i$ (resp., $L_i$) the subset of $N^+(X_i)$ consisting of those vertices that dominate (resp., are dominated by) some vertex in $N^+(X_{i+1})$ (resp., $N^+(X_{i-1})$). Moreover, let $I_i = N^+(X_i)-(L_i \cup R_i)$ and let $W_i = N^+(L_i \cup I_i \cup R_i) \cap N_2$. Note that $N^+(X_i)=L_i \cup I_i \cup R_i$ (see Figure~\ref{fig-conj}).

\begin{figure}[ht]
\centering
    \tikzset{middlearrow/.style={
	decoration={markings,
		mark= at position 0.6 with {\arrow{#1}},
	},
	postaction={decorate}
}}

\tikzset{shortdigon/.style={
	decoration={markings,
		mark= at position 0.45 with {\arrow[inner sep=10pt]{<}},
		mark= at position 0.75 with {\arrow[inner sep=10pt]{>}},
	},
	postaction={decorate}
}}

\tikzset{digon/.style={
	decoration={markings,
		mark= at position 0.4 with {\arrow[inner sep=10pt]{<}},
		mark= at position 0.6 with {\arrow[inner sep=10pt]{>}},
	},
	postaction={decorate}
}}

\def \n {5}
\def \radius {3cm}
\def \margin {8}

\begin{tikzpicture}[scale = 0.66]

    \draw (1,5) circle (20pt);
    \node (x1)  at (1,5)  {};  
	\node (label_x1)  at (1,4.0) {$X_1$};

	\draw (6,5) circle (20pt); 
	\node (x2)  at (6,5)  {};  
	\node (label_x2)  at (6,4.0) {$X_2$};

	\draw (11,5) circle (20pt);
	\node (x3)  at (11,5)  {};
	\node (label_x3)  at (11,4.0) {$X_3$};

     \draw (0,7) circle (10pt);
     \node (l1)  at (0,7)  {};
     \node (label_l1)  at (0.4,7.4) {\footnotesize{$L_1$}}; 
	 \draw (1,7) circle (10pt);
     \node (i1)  at (1,7)  {};
     \node (label_i1)  at (1.4,7.4) {\footnotesize{$I_1$}}; 	
	 \draw (2,7) circle (10pt);
     \node (r1)  at (2,7)  {};
      \node (label_r1)  at (2.4,7.4) {\footnotesize{$R_1$}};

     \draw (1,9) ellipse (1.6 and 0.6);
     \node (w11)  at (0,9)  {};
     \node (w12)  at (1,9)  {};
     \node (label_w1)  at (1,9) {$W_1$};
     \node (w13)  at (2,9)  {};

     \draw (5,7) circle (10pt);
     \node (l2)  at (5,7)  {};
     \node (label_l2)  at (5.4,7.4) {\footnotesize{$L_2$}}; 
	 \draw (6,7) circle (10pt);
	 \node (i2)  at (6,7)  {};
     \node (label_i2)  at (6.4,7.4) {\footnotesize{$I_2$}}; 
	 \draw (7,7) circle (10pt);
	 \node (r2)  at (7,7)  {};
     \node (label_r2)  at (7.4,7.4) {\footnotesize{$R_2$}};

     \draw (6,9) ellipse (1.6 and 0.6);
     \node (w21)  at (5,9)  {};
	 \node (w22)  at (6,9)  {};
	 \node (label_w2)  at (6,9) {$W_2$};
	 \node (w23)  at (7,9)  {};

     \draw (10,7) circle (10pt);
	 \node (l3)  at (10,7)  {};     
     \node (label_l3)  at (10.4,7.4) {\footnotesize{$L_3$}}; 	
	 \draw (11,7) circle (10pt);
	 \node (i3)  at (11,7)  {};     
	 \node (label_i3)  at (11.4,7.4) {\footnotesize{$I_3$}}; 
	 \draw (12,7) circle (10pt);
	 \node (r3)  at (12,7)  {};
	 \node (label_r3)  at (12.4,7.4) {\footnotesize{$R_3$}};

     \draw (11,9) ellipse (1.6 and 0.6);
     \node (w31)  at (10,9)  {};
	 \node (w32)  at (11,9)  {};
	 \node (w32)  at (11,9)  {};
	 \node (label_w3)  at (11,9) {$W_3$};
	 \node (w33)  at (12,9)  {};

	 \draw[rounded corners] (-2, 8.3) rectangle (14, 9.8) {};
	 \node (label_n2)  at (14.5,9) {$N_2$};

	 \draw[rounded corners] (-2, 6.3) rectangle (14, 7.8) {};
	 \node (label_n2)  at (14.5,7) {$N_1$};

	

 \foreach \from/\to in {x1/x2,x2/x3,x1/l1,x1/i1,x1/r1,x2/l2,x2/i2,x2/r2,r1/l2,l1/w11,i1/w12,r1/w13,l2/w21,i2/w22,r2/w23,r2/l3,x3/l3,x3/i3,x3/r3,l3/w31,i3/w32,r3/w33}
 \draw[edge,middlearrow={>}] (\from) -- (\to);

\end{tikzpicture}
\caption{\centering Illustration of sets $L_i$, $I_i$, $R_i$ e $W_i$.}
\label{fig-conj}
\end{figure}
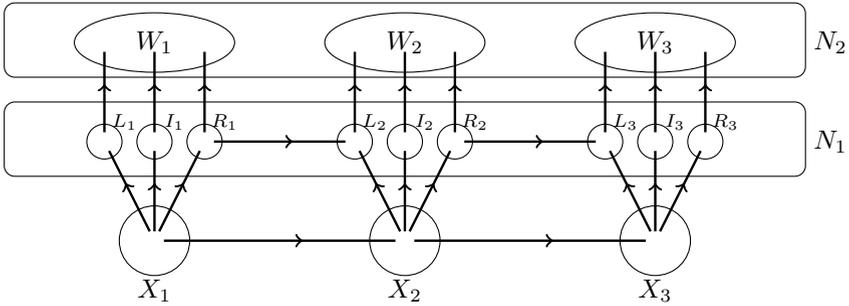

\begin{lemma}
\label{arc-in-F-prop}
Let $D$ be an arc-locally in-semicomplete digraph. Let $(V_1,V_2,V_3)$ be a partition of $V(D)$ as described in Theorem~\ref{lem-arc-resultado}\rm{(ii)}. Let $Q:=Q[X_1,X_2,\ldots, X_k]$ be the odd extended cycle of length at least five corresponding to $D[V_2]$. Then, the following hold.

\begin{enumerate}
\item[\rm{(i)}] $N_d$ is stable for all $d \geq 2$,

\item[\rm{(ii)}] there are no vertices $x_i \in X_i$, $x_j \in X_j$ and $y \in V_3$ such that $i,j \in \{1,2,\ldots,k\}$, $i \neq j$ and $\{x_i,x_j\} \to y$, 

\item[\rm{(iii)}] there are no vertices $u \in N^+(X_i)$, $v \in N^+(X_j)$ such that $i,j \in \{1,2,\ldots,k\}$, $i \neq j$, $X_i$ and $X_j$ are non-adjacent and $u \to v$,

\item[\rm{(iv)}] $N^+(X_i) \Rightarrow N^+(X_{i+1})$ for all $i \in \{1,2,\ldots,k\}$,

\item[\rm{(v)}] $N^-(N_d) \subseteq N_{d-1} \cup V_1$ for all $d \geq 1$, 

\item[\rm{(vi)}] the digraph $D[N_1]$ does not contain a path of length two,

\item[\rm{(vii)}] $N^+(X_i)$ is stable for all $i \in \{1,2,\ldots,k\}$,

\item[\rm{(viii)}] the sets $L_i$, $I_i$ and $R_i$ are pairwise disjoint, $N^-(L_i) \subseteq R_{i-1} \cup X_i \cup V_1$, $N^-(I_i \cup R_i) \subseteq X_i \cup V_1$, $N^+(R_i) \subseteq W_i \cup L_{i+1}$, $N^+(L_i \cup I_i) \subseteq W_i$ and $X_i \mapsto R_i$ for all $i \in \{1,\ldots, k\}$, 

\item[\rm{(ix)}] $N^-(W_i) \subseteq L_i \cup I_i \cup R_i \cup V_1$ for all $i \in \{1,2,\ldots,k\}$.

\end{enumerate}
\end{lemma}

\begin{proof}

\textbf{\rm{(i)}} First, towards a contradiction assume there exists an edge $uv$ with $\{u,v\} \subseteq N_2$. Let $x \in Q$ and $y \in N_1$ such that $x \to y$ and $y \to v$. Since $D$ is arc-locally in-semicomplete, it follows that $u$ and $x$ are adjacent. Since $V_2 \Rightarrow V_3$, it follows that $x \to u$, a contradiction because $u \in N_2$. Therefore, $N_2$ is a stable set. Towards a contradiction assume there exists $N_d$ with $d>2$ which is not stable. Choose such $N_d$ with $d$ as small as possible. Let $u,v \in N_d$ such that $u \to v$. Let $x,y$ be the vertices of $N_{d-1}$ that dominate $u$ and $v$, respectively. Since $D[V_3]$ is bipartite, it follows that $x \neq y$. Since $D$ is arc-locally in-semicomplete, it follows that $x$ and $y$ are adjacent, a contradiction to the choice of $d$.

\textbf{\rm{(ii)}} Towards a contradiction, assume that there are vertices $x_i \in X_i$, $x_j \in X_j$ and $y \in V_3$ such that $i,j \in \{1,2,\ldots,k\}$, $i \neq j$ and $\{x_i,x_j\} \to y$. Without loss of generality, assume that $i<j$. By Lemma~\ref{lema_arc-livre-triangulo}, $x_i$ and $x_j$ cannot be adjacent. So $X_{i-1} \neq X_j$ and $X_{j-1} \neq X_i$ where indices are taken modulo $k$. Let $x_{i-1} \in X_{i-1}$ and let $x_{j-1} \in X_{j-1}$. Since $x_jy \in E(D$), $x_i \to y$, $x_{j-1} \to x_j$ and $D$ is arc-locally in-semicomplete, it follows that $x_i \to x_{j-1}$, and hence, $i=j-2$. Using the same argument but with the roles of $X_i$ and $X_j$ exchanged, we conclude that $j=i-2$. This is a contradiction since $k \geq 5$.

\textbf{\rm{(iii)}} Towards a contradiction, assume that there are vertices $u \in N^+(X_i)$, $v \in N^+(X_j)$ such that $i,j \in \{1,2,\ldots,k\}$, $i \neq j$, $X_i$ and $X_j$ are non-adjacent and $u \to v$. Let $x_i$ be a vertex in $X_i$ that dominates $u$ and let $x_j$ be a vertex in $X_j$ that dominates $v$. Since $uv \in E(D)$, $x_i \to u$, $x_j \to v$ and $D$ is arc-locally in-semicomplete, it follows that $x_i$ and $x_j$ are adjacent, a contradiction since $X_i$ and $X_j$ are non-adjacent.

\textbf{\rm{(iv)}} Towards a contradiction, assume without loss of generality that there exists an edge $uv \in E(D)$ such that $u \in N^+(X_3)$ and $v \in N^+(X_2)$. Let $x_3$ be a vertex of $X_3$ that dominates $u$ and let $x_2$ be a vertex of $X_2$ that dominates $v$. Let $x_1 \in X_1$. Since $x_2v \in E(D)$, $u \to v$, $x_1 \to x_2$, $u \in V_3$ and $D$ is arc-locally in-semicomplete, it follows that $x_1 \to u$ contradicting \rm{(ii)}.

\textbf{\rm{(v)}} Towards a contradiction, assume that for some $d \geq 1$ and some $j \neq d$ there exists an edge $uv \in E(D)$ such that $v \in N_d$, $u \in N_j$ and $j \neq d-1$. Choose such $d$ as small as possible. By the definition of $N_d$, it follows that $j>d$. Let $y$ be a vertex in $N_{d-1}$ that dominates $v$ and let $x$ be a vertex of $N_{d-2}$ that dominates $y$, if $d \geq 2$, otherwise let $\{y,x\} \subseteq N_0=V_2$ such that $x \to y$ and $y \to v$. Since $yv \in E(D)$, $u \to v$, $x \to y$ and $D$ is arc-locally in-semicomplete, it follows that $u$ and $x$ are adjacent. By the definition of $N_j$, $u \to x$. Since $u \in V_3$ and $V_2 \Rightarrow V_3$, $x \notin V_2$; so $d \geq 3$. Thus $x \in N_{d-2}$ has an in-neighbor $u \in N_j$ with $j > d-2$, contradicting the choice of $d$.

\textbf{\rm{(vi)}} Towards a contradiction, suppose that there exists a path $P=u_1u_2u_3$ in $D[N_1]$. Let $x_i \in X_i$ be a vertex of $Q$ that dominates $u_3$. Since $u_1 \to u_2$, $u_2u_3 \in E(D)$, $x_i \to u_3$, $D$ is arc-locally in-semicomplete and $V_2 \Rightarrow V_3$, it follows that $x_i \to u_1 $. Let $x_{i-1} \in X_{i-1}$, $x_{i-2} \in X_{i-2}$ and $x_{i-3} \in X_{i-3}$ be vertices of $Q$ where indices are taken modulo $k$. So $x_{i-3} \to x_{i-2}$, $x_{i-2} \to x_{i-1}$ and $x_{i-1} \to x_{i}$. Since $x_{i-1} \to x_i$, $x_iu_3 \in E(D)$, $u_2 \to u_3$, $D$ is arc-locally in-semicomplete and $V_2 \Rightarrow V_3$, it follows that $x_{i-1} \to u_2$. Analogously for $x_{i-2}$, $u_1$ and the edge $x_{i-1}u_2$, we conclude that $x_{i-2} \to u_1$. Again, similarly for $x_{i-3}$, $x_i$ and the edge $x_{i-2}u_1$, we conclude that $x_{i-3}$ and $x_i$ are adjacent, a contradiction to the fact that $Q$ is an extended cycle of length at least five.

\textbf{\rm{(vii)}} Towards a contradiction, assume that there exists an edge $u_1u_2$ in $E(D)$ such that $\{u_1,u_2\} \subseteq N^+(X_i)$, for some $i$ in $\{1,2,\ldots,k\}$. Let $v_1$ and $v_2$ be vertices of $X_i$ that dominate $u_1$ and $u_2$, respectively. By extended cycle definition, $X_i$ is stable. By Lemma~\ref{lema_arc-livre-triangulo}, the graph $U(D[V_2 \cup V_3])$ does not contain a cycle of length three, and hence, $v_1 \neq v_2$. Since $u_1u_2 \in E(D)$, $v_1 \to u_1$, $v_2 \to u_2$ and $D$ is arc-locally in-semicomplete, it follows that $v_1$ and $v_2$ are adjacent, a contradiction to the fact that $X_i$ is a stable set.

\textbf{\rm{(viii)}} By definition, $I_i$ is disjoint from both $L_i$ and $R_i$; also by (\rm{vi}) it follows $L_i \cap R_i = \emptyset$ for all $i \in \{1,2,\ldots,k\}$. Now, towards a contradiction assume that $N^-(L_i \cup I_i \cup R_i) \not\subseteq R_{i-1} \cup X_i \cup V_1$ for some $i \in \{1,2,\ldots,k\}$. Let $v$ be a vertex in $V(D)-(R_{i-1} \cup X_i \cup V_1)$ that dominates a vertex $u$ in $L_i \cup I_i \cup R_i$. By (\rm{ii}), $v \not\in V(Q)$, and by (\rm{v}) it follows $v\not\in N_d$ for all $d\geq 2$. Thus $v \in N^+(X_j)$ for some $j \neq i$. By (\rm{iv}), $j\neq i+1$ but this contradicts (\rm{iii}); so $N^-( L_i \cup I_i \cup R_i ) \subseteq R_{i-1} \cup X_i \cup V_1$. By definition of $L_i$, $I_i$ and $R_i$, it follows that $N^-(L_i) \subseteq R_{i-1} \cup X_i \cup V_1$ and $N^-(I_i \cup R_i) \subseteq X_i \cup V_1$. No vertex in $L_i \cup I_i \cup R_i$ dominates a vertex in $N^+(X_j)$ for $j\not\in\{i-1,i+1\}$ by (\rm{iii}), nor a vertex in $N^+(X_{i-1}) \cup V(Q)$ by (\rm{iv}) and $V_2 \Rightarrow V_3$. By (\rm{vii}), $N^+(X_i)$ is stable for all $i \in \{1,2,\ldots,k\}$, and hence, $N^+(R_i) \subseteq W_i \cup L_{i+1}$ and $N^+(L_i \cup I_i) \subseteq W_i$. Finally, let $u \in X_i$ and let $v\in R_i$; we want to show that $u \to v$. Let  $w \in L_{i+1}$ such that $v \to w$. Let $x \in X_{i+1}$ such that $xw \in E(D)$; since $u \to x$, $v \to w$, $V_2 \Rightarrow V_3$ and $D$ is locally arc in-semicomplete, it follows that $u \to v$.

\textbf{\rm{(ix)}} Towards a contradiction, assume there exists $i \in \{1,2,\ldots,k\}$ such that $N^-(W_i) \not\subseteq L_i \cup I_i \cup R_i \cup V_1$. Let $v$ be a vertex in $V(D)-(L_i \cup I_i \cup R_i \cup V_1)$ that dominates a vertex $w$ in $W_i$. By (\rm{v}), $v \in N^+(X_j)$ for some $j \neq i$. Let $x_j$ be a vertex in $X_j$ such $x_j \to v$ and let $u$ be a vertex in $L_i \cup I_i \cup R_i$ such that $u \to w$. Since $V(Q) \Rightarrow V_3$, $vw \in E(D)$, $x_j \to v$, $u \to w$ and $D$ is arc-locally in-semicomplete, it follows that $x_j \to u$. Let $x_i$ be a vertex in $X_i$ such that $x_i \to u$. Thus $\{x_i,x_j\} \to u$ which contradicts (\rm{ii}).
\end{proof}

In next sections we verify Conjectures~\ref{conj_berge} and \ref{conj_be} for the class of arc-locally (out) in-semicomplete digraphs. First, let $D$ be an arc-locally in-semicomplete digraph and let $H$ be the inverse of $D$. Note that $H$ is an arc-locally out-semicomplete digraph. By definition, an $S_{BE}$-path partition (resp., $S$-path partition) of $D$ is also an $S_{BE}$-path partition (resp., $S$-path partition) in $H$, but with the direction of the paths inverted. Thus $D$ satisfies the BE-property (resp., $\alpha$-property) if and only if $H$ satisfies the BE-property (resp., $\alpha$-property). So from now on we aim to prove Conjectures~\ref{conj_berge} and \ref{conj_be} for arc-locally in-semicomplete digraphs. Moreover, Sambinelli, Silva and Lee~\cite{tesemaycon2018,ssl} proved the following lemmas.

\begin{lemma}[Sambinelli, Silva and Lee, 2018]
\label{samb_part}
Let $D$ be a digraph. If $V(D)$ can be partitioned into $k$ subsets, say $V_1,V_2 , \ldots , V_{k>2}$, such that $D[V_i]$ satisfies the BE-property (resp., $\alpha$-property) and $\alpha(D) = \sum_{i=1}^k \alpha(D[V_i])$, then $D$ satisfies the BE-property (resp., $\alpha$-property).
\end{lemma}

\begin{lemma}[Sambinelli, Silva and Lee, 2018]
\label{samb_clique}
Let $D$ be a digraph. If $D$ has a clique cut, then $V(D)$ can be partitioned into two subsets $V_1$ and $V_2$ such that $\alpha(D) = \alpha(D[V_1]) + \alpha(D[V_2])$.
\end{lemma}

Thus by Lemmas~\ref{samb_part} and \ref{samb_clique} if a digraph $D$ is a minimal counterexample for Conjectures~\ref{conj_berge} or \ref{conj_be}, then $D$ is connected and $D$ has no clique cut. Moreover, by Lemmas~\ref{diper-alpha} and \ref{diperD-BE}, Conjectures~\ref{conj_berge} and \ref{conj_be} hold for diperfect digraphs. So we may assume that $D$ is connected, not diperfect and has no clique cut. Therefore, $V(D)$ admits a partition as described in Theorem~\ref{lem-arc-resultado}\rm{(ii)}.

\section{Begin-End Conjecture}
\label{be-alis}

In this section we prove that Conjecture~\ref{conj_be} holds for arc-locally (out) in-semicomplete digraphs. Recall that $\D$ denotes the set of all digraphs containing no induced blocking odd cycle. 

First we present an outline of the main proof. Let $D$ be an arc-locally in-semicomplete digraph. Note that every induced subdigraph of $D$ is also an arc-locally in-semicomplete digraph. Thus, it is suffices to show that $D$ satisfies the BE-property. By Theorem~\ref{lem-arc-resultado}\rm{(ii)}, $V(D)$ admits a partition $(V_1,V_2,V_3)$ as described in the statement. First, we show that if $D \in \D$, then $V_1=\emptyset$. Next, we show that an extended cycle satisfies the BE-property. Finally, we show that if $V_3 \neq \emptyset$, then $D$ satisfies the BE-property. This last case is divided into two subcases, depending on whether there exists a vertex $v$ in $V_3$ such that $\textrm{dist}(V_2,v) \geq 3$ or not.

\begin{lemma}
\label{lema-arc-V1-empty}
Let $D$ be an arc-locally in-semicomplete digraph. Let $(V_1,V_2,V_3)$ be a partition of $V(D)$ as described in Theorem~\ref{lem-arc-resultado}\rm{(ii)}. If $D \in \D$, then $V_1=\emptyset$. 
\end{lemma} 

\begin{proof}
Towards a contradiction, assume that there exists $v$ in $V_1$. Let $xy$ be an edge of $D[V_2]$. Since $V_1 \mapsto V_2$ and $D[V_2]$ is an extended cycle, it follows that $D[\{v,x,y\}]$ is a transitive triangle, a contradiction to the fact that $D \in \D $.
\end{proof}

Next, we prove that an extended cycle satisfies the BE-property. 

\begin{lemma}
\label{lema_circuito_estendido}
If a digraph $D \in \D$ is an extended cycle, then $D$ satisfies the BE-property.
\end{lemma}

\begin{proof}
Let $D:= D[X_1, X_2, \dots, X_k]$ be an extended cycle and let $S$ be a maximum stable set of $D$. Recall that $X_1 \mapsto X_2 \mapsto \cdots \mapsto X_k \mapsto X_1$. If $k$ is even, then $D$ is a bipartite digraph. Since a bipartite digraph is diperfect, the result follows by Lemma~\ref{diperD-BE}. Thus, we may assume that $k$ is odd. Note that for each $X_i$, it follows that $ X_i \cap S = \emptyset $ or $ X_i \subseteq S$, because $X_i \mapsto X_{i + 1}$ for all $ i \in \{1,2,\ldots, k\} $. Also, if $X_i \cap S = X_i$, then $ X_{i + 1} \cap S = X_{i-1} \cap S = \emptyset$. Since $k$ is odd, there exists some $i$ such that $ X_i \cap S = X_ {i + 1}\cap S = \emptyset$. Now we proceed to prove the result by induction on $\vert V(D)\vert $.

If $D$ is an odd cycle, that is, each $X_i$ is singleton, then the result follows easily. Without loss of generality, assume that $X_1 \subseteq S$ and $X_2 \cap S = X_3 \cap S = \emptyset$. Let $P = x_1x_2x_3$ be a path with $x_i \in X_i $ for $ i \in \{1,2,3\} $. Let $D'= D -\{x_1,x_2,x_3\}$ and let $S'= S-x_1$. We show next that $S'$ is a maximum stable set in $D'$. Towards a contradiction, assume that $S'$ is not a maximum stable set in $D'$ and let $Z$ be a maximum stable set in $D'$. So $\vert Z\vert  > \vert S'\vert  = \vert S\vert -1$, and this implies that $\vert Z\vert  = \vert S\vert $. Note that $Z$ must necessarily contain one of the sets $X_i-x_i$, $i \in \{1,2,3\}$, otherwise $Z \cup X_2$ would be a stable set in $D$ larger than $S$, a contradiction. Assume that $X_i-x_i \subseteq Z$ for some $i \in \{1,2,3\}$. Thus $Z \cup x_i$ is a stable set larger than $S$ in $D$, a contradiction. Therefore, $S'$ is maximum in $D'$. If $D'$ is disconnected, then $D'$ is bipartite, and hence, satisfies the BE-property. If $D'$ is connected, then $D'$ is an odd extended cycle with $\vert V(D')\vert  < \vert V(D)\vert $, and by induction hypothesis $D'$ satisfies the BE-property. Let $\sP'$ be an $S'_{BE}$-path partition of $D'$. Thus $ \sP' \cup P$ is an $S_{BE}$-path partition $D$. This finishes the proof.    
\end{proof}

Let $D$ be an arc-locally in-semicomplete digraph. Let $(V_1,V_2,V_3)$ be a partition of $V(D)$ as described in Theorem~\ref{lem-arc-resultado}\rm{(ii)}. Let $Q:=Q[X_1,X_2, \ldots, X_k]$ be the odd extended cycle of length at least five corresponding to $D[V_2]$. Recall that $N_d$ is the set of vertices that are at distance $d$ from $Q$, $R_i$ (resp., $L_i$) the subset of $N^+(X_i)$ consisting of those vertices that dominate (resp., are dominated by) some vertex in $N^+(X_{i+1})$ (resp., $N^+(X_{i-1})$). Moreover, $I_i = N^+(X_i)-(L_i \cup R_i)$ and $W_i = N^+(L_i \cup I_i \cup R_i) \cap N_2$.

\begin{lemma}
\label{arc-in-case-d<=2}
Let $D$ be an arc-locally in-semicomplete digraph such that every proper induced subdigraph of $D$ satisfies the BE-property. Let $(V_1,V_2,V_3)$ be a partition of $V(D)$ as described in Theorem~\ref{lem-arc-resultado}\rm{(ii)}. If $N_d = \emptyset$ for $d \geq 3$ and $V_1=\emptyset$, then $D$ satisfies the BE-property. 
\end{lemma}

\begin{proof}

Let $Q:=Q[X_1,X_2, \ldots, X_k]$ be the odd extended cycle of length at least five corresponding to $D[V_2]$. Let $S$ be a maximum stable set of $D$. By hypothesis, $N^+(N_2)=\emptyset$ and $V_1 = \emptyset$. By Lemma~\ref{arc-in-F-prop}\rm{(i)} and \rm{(vii)}, it follows that $W_i$ and $L_i \cup I_i \cup R_i$ are stable. Next, we prove some claims. \newline

\textbf{Claim 1.} We may assume that $N^+(L_i)=\emptyset$ for all $i \in \{1,\ldots,k\}$. 

Assume that there exists $i \in \{1,2,\ldots,k\}$ such that $N^+(L_i) \neq \emptyset$. By Lemma~\ref{arc-in-F-prop}\rm{(viii)}, $N^+(L_i) \subseteq W_i$ and by Lemma~\ref{arc-in-F-prop}\rm{(ix)} it follows that $N^-(W_i) \subseteq L_i \cup I_i \cup R_i$. Let $H:=H[X,Y]$ be a maximal connected bipartite subdigraph with edges between $L_i$ and $N^+(L_i)$. Assume that $X \subseteq L_i$ and $Y \subseteq N^+(L_i) \subseteq W_i$. Since $Y \subseteq W_i$, it follows by Lemma~\ref{arc-in-F-prop}\rm{(v)} that $X \Rightarrow Y$. By Lemma~\ref{arc-in-F-prop}\rm{(viii)}, the sets $L_i$, $I_i$ and $R_i$ are disjoint. Towards a contradiction, assume that there exists $v \in I_i \cup R_i$ such that $v$ dominates a vertex $u$ in $Y$. Let $x \in X$ and $y \in R_{i-1}$ be vertices such that $x \to u$ and $y \to x$. Since $v \to u$ and $D$ is arc-locally in-semicomplete, we have that $y$ and $v$ are adjacent, and by Lemma~\ref{arc-in-F-prop}\rm{(iv)} it follows that $y \to v$, a contradiction to fact that $v \notin L_i$. Since $H$ is maximal and connected, $Y \subseteq W_i$ and $N^-(W_i) \subseteq L_i \cup I_i \cup R_i$, it follows that $N^-(Y) = X \subseteq L_i$. Let $U=N^-(X)$. By Lemma~\ref{arc-in-F-prop}\rm{(viii)}, $U \subseteq R_{i-1} \cup X_i$. By Lemma~\ref{arc-in-F-prop}\rm{(iv)} and $V_2 \Rightarrow V_3$, it follows that $U \Rightarrow X$. By Lemma~\ref{lem_arc_1} applied to $U$ and $H$, $U \mapsto X$. Since $N^+(Y)=\emptyset$, $N(Y)=X$. Since $X$ and $Y$ are stable, $N(Y) = X$, $N(X)= U \cup Y$ and every vertex in $U$ is adjacent to every vertex in $X$, it follows by Lemma~\ref{arc-in-par-W-B-U} applied to $U$, $X$ and $Y$ that $D$ has an $S_{BE}$-path partition. So we may assume that $N^+(L_i) = \emptyset$ for all $i \in \{1,2,\ldots,k\}$. This ends the proof of Claim 1.  \newline

From now on, let $I^+_i= N^-(W_i) \cap I_i$ for all $i \in \{1,2,\ldots,k\}$. The Figure~\ref{fig-proof} illustrates the structure of $D$ applying Claim 1 and Lemma~\ref{arc-in-F-prop}. \newline
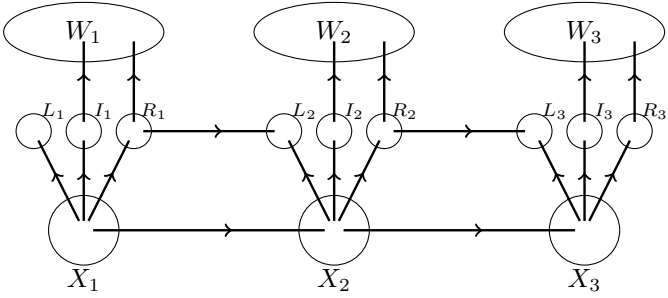
\begin{figure}[ht]
\centering
    \tikzset{middlearrow/.style={
	decoration={markings,
		mark= at position 0.6 with {\arrow{#1}},
	},
	postaction={decorate}
}}

\tikzset{shortdigon/.style={
	decoration={markings,
		mark= at position 0.45 with {\arrow[inner sep=10pt]{<}},
		mark= at position 0.75 with {\arrow[inner sep=10pt]{>}},
	},
	postaction={decorate}
}}

\tikzset{digon/.style={
	decoration={markings,
		mark= at position 0.4 with {\arrow[inner sep=10pt]{<}},
		mark= at position 0.6 with {\arrow[inner sep=10pt]{>}},
	},
	postaction={decorate}
}}

\def \n {5}
\def \radius {3cm}
\def \margin {8}

\begin{tikzpicture}[scale = 0.66]

    \draw (1,5) circle (20pt);
    \node (x1)  at (1,5)  {};  
	\node (label_x1)  at (1,4.0) {$X_1$};

	\draw (6,5) circle (20pt); 
	\node (x2)  at (6,5)  {};  
	\node (label_x2)  at (6,4.0) {$X_2$};

	\draw (11,5) circle (20pt);
	\node (x3)  at (11,5)  {};
	\node (label_x3)  at (11,4.0) {$X_3$};

     \draw (0,7) circle (10pt);
     \node (l1)  at (0,7)  {};
     \node (label_l1)  at (0.4,7.4) {\footnotesize{$L_1$}}; 
	 \draw (1,7) circle (10pt);
     \node (i1)  at (1,7)  {};
     \node (label_i1)  at (1.4,7.4) {\footnotesize{$I_1$}}; 	
	 \draw (2,7) circle (10pt);
     \node (r1)  at (2,7)  {};
      \node (label_r1)  at (2.4,7.4) {\footnotesize{$R_1$}};

     \draw (1,9) ellipse (1.6 and 0.6);
     \node (w11)  at (0,9)  {};
     \node (w12)  at (1,9)  {};
     \node (label_w1)  at (1,9) {$W_1$};
     \node (w13)  at (2,9)  {};

     \draw (5,7) circle (10pt);
     \node (l2)  at (5,7)  {};
     \node (label_l2)  at (5.4,7.4) {\footnotesize{$L_2$}}; 
	 \draw (6,7) circle (10pt);
	 \node (i2)  at (6,7)  {};
     \node (label_i2)  at (6.4,7.4) {\footnotesize{$I_2$}}; 
	 \draw (7,7) circle (10pt);
	 \node (r2)  at (7,7)  {};
     \node (label_r2)  at (7.4,7.4) {\footnotesize{$R_2$}};

     \draw (6,9) ellipse (1.6 and 0.6);
     \node (w21)  at (5,9)  {};
	 \node (w22)  at (6,9)  {};
	 \node (label_w2)  at (6,9) {$W_2$};
	 \node (w23)  at (7,9)  {};

     \draw (10,7) circle (10pt);
	 \node (l3)  at (10,7)  {};     
     \node (label_l3)  at (10.4,7.4) {\footnotesize{$L_3$}}; 	
	 \draw (11,7) circle (10pt);
	 \node (i3)  at (11,7)  {};     
	 \node (label_i3)  at (11.4,7.4) {\footnotesize{$I_3$}}; 
	 \draw (12,7) circle (10pt);
	 \node (r3)  at (12,7)  {};
	 \node (label_r3)  at (12.4,7.4) {\footnotesize{$R_3$}};

     \draw (11,9) ellipse (1.6 and 0.6);
     \node (w31)  at (10,9)  {};
	 \node (w32)  at (11,9)  {};
	 \node (w32)  at (11,9)  {};
	 \node (label_w3)  at (11,9) {$W_3$};
	 \node (w33)  at (12,9)  {};



	

 \foreach \from/\to in {x1/x2,x2/x3,x1/l1,x1/i1,x1/r1,x2/l2,x2/i2,x2/r2,r1/l2,i1/w12,r1/w13,i2/w22,r2/w23,r2/l3,x3/l3,x3/i3,x3/r3,i3/w32,r3/w33}
 \draw[edge,middlearrow={>}] (\from) -- (\to);

\end{tikzpicture}
\caption{\centering By Claim 1 and Lemma~\ref{arc-in-F-prop} $D$ has this structure: the sets $L_i$, $I_i$, $R_i$, $W_i$ and $X_i$ are stable, $L_i \cap I_i = \emptyset$, $I_i \cap R_i =\emptyset$, $L_i \cap R_i = \emptyset$, $N^-(W_i) \subseteq I^+_i \cup R_i$, $N^-(L_i) \subseteq R_{i-1} \cup X_i$, $N^-(I_i \cup R_i) = X_2$, $N^+(L_i)=\emptyset$, $N^+(I^+_i)\subseteq W_i$, $N^+(I_i - I^+_i)=\emptyset$ and $N^+(R_i) \subseteq W_i \cup L_{i+1}$.}
\label{fig-proof}
\end{figure}

\textbf{Claim 2.} We may assume that $X_i \mapsto I^+_i \cup R_i \cup X_{i+1}$ for all $i \in \{1,2,\ldots,k\}$. 

Let $i$ in $\{1,2,\ldots,k\}$. Since $V(Q) \Rightarrow V_3$ and $Q$ is an extended cycle, it follows that $X_i \Rightarrow I^+_i \cup R_i$ and $X_i \mapsto X_{i+1}$. By Lemma~\ref{arc-in-F-prop}\rm{(viii)}, $X_i \mapsto R_i$. So it remains to show that $X_i \mapsto I^+_i$. Since $V_1=\emptyset$, it follows by Lemma~\ref{arc-in-F-prop}\rm{(ix)} and Claim 1 that $N^-(W_i) \subseteq I_i \cup R_i$. Let $H:=H[X,Y]$ be a maximal connected bipartite subdigraph with edges between $I^+_i \cup R_i$ and $W_i$. Assume that $X \subseteq I^+_i \cup R_i$ and $Y \subseteq W_i$. Let $U=N^-(X)$. Since $Y \subseteq W_i$, it follows by Lemma~\ref{arc-in-F-prop}\rm{(v)} that $X \Rightarrow Y$. Since $V(Q) \Rightarrow V_3$ and $X \Rightarrow Y$, it follows by Lemma~\ref{lem_arc_1} applied to $U$ and $H$ that $U \mapsto X$. Since $H$ is maximal and connected, if $X \subseteq I^+_i$, then $N(X)= U \cup Y$ and $N(Y)=X$, and hence, it follows by Lemma~\ref{arc-in-par-W-B-U} applied to $U$, $X$ and $Y$ that $D$ has an $S_{BE}$-path partition. Thus, we may assume that $X \subseteq I^+_i \cup R_i$ and $X \not\subset I^+_i$. Since $X_i \mapsto R_i$, it follows that $U=X_i$, and hence, $X_i \mapsto X$. Since $H$ is arbitrary, it follows that $X_i \mapsto I^+_i$. So we may assume that $X_i \mapsto I^+_i \cup R_i \cup X_{i+1}$ for all $i \in \{1,2,\ldots,k\}$. This ends the proof of Claim 2. \newline

\textbf{Claim 3.} We may assume that if $S \cap X_i \neq \emptyset$, then $X_i \subseteq S$ for all $i \in \{1,2,\ldots,k\}$. 

Assume that there exists $i \in \{1,2,\ldots,k\}$ such that $X_i \cap S \neq \emptyset$ and $X_i \not\subseteq S$. Without loss of generality, assume that $i=2$. By Claim 2,  $X_2 \mapsto I^+_2 \cup R_2 \cup X_3$. Since $X_1 \mapsto X_2$, it follows that $(X_1 \cup I^+_2 \cup R_2 \cup X_3 ) \cap S = \emptyset$. Let $S_1 = S \cap (L_2 \cup (I_2-I^+_2))$ and let $S_2 = S \cap W_1$. Since $X_2-S \neq \emptyset$ and $S$ is a maximum stable set, $S_1$ must be non-empty. By Claim 1, $N^+(L_1)=N^+(L_2)=\emptyset$. By hypothesis, $N^+(W_1 \cup W_2)=\emptyset$ and $V_1 = \emptyset$.  By Lemma~\ref{arc-in-F-prop}\rm{(ix)}, $N^-(W_1) \subseteq I_1 \cup R_1$. By definition of $I^+_2$ and by Lemma~\ref{arc-in-F-prop}\rm{(viii)}, we have that $N(I_2-I^+_2) \subseteq X_2$ and $N^-(L_2) \subseteq R_1 \cup X_2$. Thus, $N(S_1 \cup S_2) \subseteq I_1 \cup R_1 \cup X_2$. Since $S$ is maximum and $(X_1 \cup X_3 \cup I^+_2 \cup R_2) \cap S = \emptyset$, we have $\vert S_1 \cup S_2\vert  \geq \vert N(S_1 \cup S_2)\vert $. By Lemma~\ref{stable_set_S_menor_vizinhanca} applied to $S_1 \cup S_2$ it follows that $D$ satisfies the BE-property. So we may assume that if $S \cap X_i \neq \emptyset$, then $X_i \subseteq S$ for all $i \in \{1,2,\ldots,k\}$. This ends the proof of Claim 3. \newline

\textbf{Claim 4.} We may assume that there exists no $i \in \{1,2,\ldots,k\}$ such that $(X_i \cup X_{i+1} \cup X_{i+2}) \cap S = \emptyset$.

Without loss of generality, assume that $i=1$. Since $X_1 \mapsto X_2 \mapsto X_3$ and $S$ is maximum, it follows that $(L_2 \cup I_2 \cup R_2) \cap S \neq \emptyset$. Let $S_1 = S \cap (L_2 \cup I_2 \cup R_2)$ and let $S_2 = S \cap W_1$. By Claim 1, $N^+(L_1)=N^+(L_2)=N^+(L_3)= \emptyset$. By hypothesis, $N^+(W_1 \cup W_2)=\emptyset$ and $V_1 = \emptyset$. By Lemma~\ref{arc-in-F-prop}\rm{(ix)}, $N^-(W_1) \subseteq I_1 \cup R_1$ and $N^-(W_2) \subseteq I_2 \cup R_2$. By Lemma~\ref{arc-in-F-prop}\rm{(viii)}, $N(L_2 \cup I_2 \cup R_2) \subseteq R_1 \cup X_2 \cup W_2 \cup L_3$ and $N(I_1 \cup R_1) \subseteq X_1 \cup W_1 \cup L_2$. Thus $N(S_1 \cup S_2) \subseteq I_1 \cup R_1 \cup X_2 \cup W_2 \cup L_3$. Since $S$ is maximum and $(X_1 \cup X_2 \cup X_3) \cap S = \emptyset$, we have $\vert S_1 \cup S_2\vert  \geq \vert N(S_1 \cup S_2)\vert $, and hence, by Lemma~\ref{stable_set_S_menor_vizinhanca} applied to $S_1 \cup S_2$ it follows that $D$ satisfies the BE-property. So we may assume that there exists no $i \in \{1,2,\ldots,k\}$ such that $(X_i \cup X_{i+1} \cup X_{i+2}) \cap S = \emptyset$. This ends the proof of Claim 4. \newline

Since $Q$ is odd, there exists $i \in \{1,\ldots,k\}$ such that $(X_{i} \cup X_{i+1}) \cap S = \emptyset$. Without loss of generality, assume that $(X_2 \cup X_3) \cap S = \emptyset$. By Claim 3 and 4, it follows that $X_1 \cup X_4 \subseteq S$. By Claim 1, $N^+(L_2)=\emptyset$. Since $X_1 \subseteq S$, $(X_2 \cup X_3) \cap S = \emptyset$, we conclude that $(L_1 \cup I_1 \cup R_1) \cap S = \emptyset$ and $W_1 \cup L_2 \subseteq S$. The rest of the proof is divided into two cases, depending on whether $R_2 \neq \emptyset$ or $R_2=\emptyset$.  \newline

\textbf{Case 1.} $R_2 \neq \emptyset$. First, assume that $(I^+_2 \cup R_2) \cap S  \neq \emptyset$. Let $H:=H[X,Y]$ be a maximal connected bipartite subdigraph with edges between $(I^+_2 \cup R_2) \cap S$ and $W_2 \cup L_3$. Assume that $Y \subseteq (I^+_2 \cup R_2) \cap S$ and $X \subseteq W_2 \cup L_3$. By hypothesis and by Claim 1, we have $N^+(W_2 \cup L_3)= \emptyset$. By Lemma~\ref{arc-in-F-prop}\rm{(viii)} and \rm{(ix)}, it follows that $N(X) \subseteq I^+_2 \cup R_2 \cup X_3$ and $N(Y) \subseteq X \cup X_2$. Note that $N(X) \cap N(Y) = \emptyset$. Since $X_3 \cap S = \emptyset$ and $H$ is maximal and connected, $N(X) \cap S = Y$. By Claim 2, $X_2 \mapsto (I^+_2 \cup R_2 \cup X_3)$, and hence, every vertex in $N(Y)-X$ is adjacent to every vertex in $N(X)$. Thus by Lemma~\ref{arc-in-H-W-U} applied to $H$ it follows that $D$ has an $S_{BE}$-path partition.

So we may assume that $(I^+_2 \cup R_2) \cap S = \emptyset$. Since $(X_2 \cup X_3) \cap S = \emptyset$, it follows that $W_2 \cup L_3 \subseteq S$, $I_2-I^+_2 \subseteq S$ and $I_3-I^+_3 \subseteq S$. Now, let $X:=W_2 \cup L_3 \cup (I_3-I^+_3)$ and let $Y=N(X)$. Note that $X \subseteq S$ and $Y \cap S = \emptyset$. By Lemma~\ref{arc-in-F-prop}\rm{(viii)} and \rm{(ix)}, $Y \subseteq I^+_2 \cup R_2 \cup X_3$. Let $H=D[X \cup Y]$ be a bipartite subdigraph of $D$. Note that $X,Y$ is a bipartition of $H$. Since $X \subseteq S$ and $Y=N(X)$, by Lemma~\ref{arc-in-n-cobre-S} we may assume that there exists a matching between $X$ and $Y$ covering $X$. We show next that $M$ covers $I^+_2 \cup R_2$. By Lemma~\ref{arc-in-F-prop}\rm{(viii)}, $N^+(I^+_2 \cup R_2) = W_2 \cup L_3$. Thus by Lemma~\ref{hall-x-y-emp} applied to $U(H)$ there exists a matching $M$ between $X$ and $Y$ covering $X$ such that the restriction of $M$ on $U(H[I^+_2 \cup R_2 \cup W_2 \cup L_3])$ is a maximum matching. Since $(X_2 \cup I^+_2 \cup R_2) \cap S = \emptyset$, $N(I^+_2 \cup R_2) \cap S = W_2 \cup L_3$. Thus by Lemma~\ref{arc-in-matching-N(S)} there exists a matching between $I^+_2 \cup R_2$ and $W_2 \cup L_3$ covering $I^+_2 \cup R_2$, and this implies that $M$ covers $I^+_2 \cup R_2$.

Let $D'=D-V(M)$ and let $S'=S-X$. Since $M$ covers $X$ and $I^+_2 \cup R_2$, we have $V(D') \cap (X \cup I^+_2 \cup R_2) = \emptyset$. Assume that $S'$ is not a maximum stable set in $D'$ and let $Z$ be a maximum stable set in $D'$. So $\vert Z\vert >\vert S'\vert =\vert S\vert -\vert X\vert =\vert S\vert -\vert V(M) \cap Y\vert $. By Claim 2, $X_3 \mapsto I^+_3 \cup R_3 \cup X_4$. Since $X_2 \mapsto X_3$, if $Z \cap (X_2 \cup I^+_3 \cup R_3 \cup X_4) \neq \emptyset$, then $X_3 \cap Z = \emptyset$. Since $(I^+_2 \cup R_2) \cap V(D') = \emptyset$, $Z \cup X$ is a stable set larger than $S$ in $D$, a contradiction. So we may assume that $Z \cap (X_2 \cup I^+_3 \cup R_3 \cup X_4) = \emptyset$. Thus $Z \cup (V(M) \cap Y)$ is a stable set larger than $S$ in $D$, a contradiction. Therefore, $S'$ is a maximum stable set in $D'$. Let $\sP_{M}$ be the set of paths in $D$ corresponding to the edges of $M$. By hypothesis, $D'$ is DE-diperfect. Let $\sP'$ be an $S'_{BE}$-path partition of $D'$. Thus the collection $\sP' \cup \sP_M$ is an $S_{BE}$-path partition of $D$.  \newline

\textbf{Case 2.} $R_2 = \emptyset$. First, we prove that $W_2 = \emptyset$. By Claim 2, $X_2 \mapsto I^+_2 \cup X_3$. By Claim 1, $N^+(L_2)=\emptyset$. Suppose that $W_2 \neq \emptyset$. Let $H:=H[X,Y]$ be a maximal connected bipartite subdigraph with edges between $I^+_2$ and $W_2$. Assume that $X \subseteq I^+_2$ and $Y \subseteq W_2$. Since $R_2 = \emptyset$ and $H$ is maximal and connected, we conclude that $N(Y)=X$ and $N(X)= X_2 \cup Y$. Since $X_2 \mapsto X$, it follows by Lemma~\ref{arc-in-par-W-B-U} applied to $X_2$, $X$ and $Y$ that $D$ has an $S_{BE}$-path partition. So we may assume that $W_2 = \emptyset$. Since $X_1 \subseteq S$ and $(X_2 \cup X_3) \cap S = \emptyset$, it follows that $X_1 \cup W_1 \cup L_2 \cup I_2 \subseteq S$.

Let $X:=W_1 \cup L_2 \cup I_2 \cup X_3$ and let $Y=N^-(X)$. Note that $X \neq \emptyset$ because $X_3 \neq \emptyset$. By Lemma~\ref{arc-in-F-prop}\rm{(viii)} and \rm{(ix)}, it follows that $Y=I^+_1 \cup R_1 \cup X_2$ and $Y \Rightarrow X$. Let $H=D[X \cup Y]$ be a bipartite subdigraph of $D$. Note that $X,Y$ is a bipartition of $H$. Since $N^+(W_1)= \emptyset$ and $W_2= \emptyset$, we conclude that $N^+(X \cap S)=\emptyset$. By Claim 2, $X_1 \mapsto I^+_1 \cup R_1 \cup X_2$. Since $X_1 \subseteq S$, $Y \cap S = \emptyset$. Thus by Lemma~\ref{bip-x-y-font} applied to $H$ we may assume that there exists a matching $M$ between $X$ and $Y$ covering $X$.

Let $D'=D-X_3$. Since $X_3 \cap S= \emptyset$, it follows that $S$ is maximum in $D'$. By hypothesis, $D'$ is BE-diperfect. Let $\sP'$ be an $S_{BE}$-path partition of $D'$. Let $\sP_M$ be the set of paths corresponding to the edges of $M \cap E(D')$. If $\sP_M= \emptyset$, then $W_1 \cup I^+_1 \cup R_1 \cup L_2 \cup I_2 = \emptyset$. Since $X_2 \cap S = \emptyset$, every vertex in $X_2$ ends a path in $\sP'$. Since $X_2 \mapsto X_3$ and $M$ covers $X_3$, it is easy to see that using the edges of $M$, we can add the vertices of $X_3$ to paths in $P'$ that ends at some vertex in $X_2$, obtaining an $S_{BE}$-path partition of $D$. So we may assume that $\sP_M \neq \emptyset$. Let $\sP_Y$ be the set of paths of $\sP'$ such that $V(P) \cap Y \neq \emptyset$ for all $P \in \sP_Y$. Since $X_1 \subseteq S$, $(X-X_3) \subseteq S$ and $Y$ is a stable set, it follows that every path in $\sP_Y$ has length one. Moreover, every $P \in \sP_Y$ starts at some vertex of $X_1$ or ends at some vertex of $X-X_3$. Let $\sP^* = (\sP'-\sP_Y) \cup \sP_M$. Note that every vertex of $X-X_3$ is an end of some path in $\sP^*$. Also, note that there might be some vertex of $Y$ which does not belong to any path in $\sP^*$. Since $\sP'$ is an $S_{BE}$-path partition of $D'$, every vertex in $Y$ belongs to some path of $\sP'$ and since every vertex in $X-X_3$ belongs to some path in $\sP^*$, there are at least $\vert Y-V(M)\vert $ vertices in $X_1$ that do not belong to any path of $\sP^*$. Since $X_1 \mapsto I^+_1 \cup R_1 \cup X_2$, we can add to $\sP^*$ the path $u \to v$ where $v \in Y-V(M)$ and $u$ is a vertex in $X_1$ that does not belong to any path of $\sP$. Since $M$ covers $X_3$, there are at least $\vert X_3\vert $ paths in $\sP^*$ that end in vertices of $X_2$. So it easy to see that using the edges of $M$, we can add the vertices of $X_3$ to paths in $\sP^*$ that ends at some vertex in $X_2$, obtaining an $S_{BE}$-path partition of $D$. This finishes the proof.
\end{proof}

\begin{lemma}
\label{arc-in-case-d=3}
Let $D$ be an arc-locally in-semicomplete digraph such that every proper induced subdigraph of $D$ satisfies the BE-property. Let $(V_1,V_2,V_3)$ be a partition of $V(D)$ as described in Theorem~\ref{lem-arc-resultado}\rm{(ii)}. If $N_d \neq \emptyset$ for some $d \geq 3$ and $V_1=\emptyset$, then $D$ satisfies the BE-property. 
\end{lemma}

\begin{proof}
Let $N_d \neq \emptyset$ such that $d$ is maximum. Note that $d \geq 3$. By Lemma~\ref{arc-in-F-prop}\rm{(i)}, the sets $N_{d}$ and $N_{d-1}$ are stable. Since $V_1 = \emptyset$, it follows by Lemma~\ref{arc-in-F-prop}\rm{(v)} that $N^-(N_d) \subseteq N_{d-1}$, $N^-(N_{d-1}) \subseteq N_{d-2}$ and $N^-(N_{d-2}) \subseteq N_{d-3}$, this implies that $N_{d-2} \Rightarrow N_{d-1}$ and $N_{d-1} \Rightarrow N_d$. Let $H:=H[X,Y]$ be a maximal connected bipartite subdigraph with edges between $N_{d-1}$ and $N_d$. Assume that $X \subseteq N_{d-1}$ and $Y \subseteq N_d$. Let $U=N^-(X)$. Since $U \Rightarrow X$ and $X \Rightarrow Y$, it follows by Lemma~\ref{lem_arc_1} applied to $U$ and $H$ that $U \mapsto X$. By the choice of $d$, $N^+(Y)=\emptyset$. Since $H$ is maximal and connected, we conclude that $N(Y)=X$ and $N(X)= U \cup Y$. Since $U \mapsto X$, it follows by Lemma~\ref{arc-in-par-W-B-U} applied to $U$, $X$ and $Y$ that $D$ has an $S_{BE}$-path partition. 
\end{proof}

Now, we ready for the main result of this section.

\begin{theorem}
\label{arc-in-BE}
Let $D$ be an arc-locally in-semicomplete digraph. If $D \in \D$, then $D$ is BE-diperfect.
\end{theorem}

\begin{proof}
Since every induced subdigraph of $D$ is also an arc-locally in-semicomplete digraph, it suffices to show that $D$ satisfies the BE-property. If $D$ is diperfect or $D$ has a clique cut, then the result follows by Lemmas~\ref{diperD-BE}, \ref{samb_part} and \ref{samb_clique}. So we may assume that $V(D)$ can be partitioned into $(V_1, V_2, V_3)$ as described in Theorem~\ref{lem-arc-resultado}\rm{(ii)}. By Lemma~\ref{lema-arc-V1-empty}, $V_1 = \emptyset$. If $V_3 = \emptyset$, then the result follows by Lemma~\ref{lema_circuito_estendido}. Then, by Lemmas~\ref{arc-in-case-d<=2} and \ref{arc-in-case-d=3} it follows that $D$ satisfies the BE-property. This finishes the proof.
\end{proof}

Let $D$ be an arc-locally in-semicomplete digraph and let $H$ be the inverse of $D$. Since $D$ satisfies the BE-property if and only if $H$ satisfies the BE-property, we have the following result.

\begin{theorem}
\label{arc-out-BE}
Let $D$ be an arc-locally out-semicomplete digraph. If $D \in \D$, then $D$ is BE-diperfect.
\qed
\end{theorem}

\section{Berge's conjecture}
\label{berge-alis}

In this section we prove that Conjecture~\ref{conj_berge} holds for arc-locally (out) in-semicomplete digraphs. Recall that we denote by $\sB$ the set of all digraphs containing no induced anti-directed odd cycle. 

First we present an outline of the main proof. Let $D$ be an arc-locally in-semicomplete digraph. Since every induced subdigraph of $D$ is also an arc-locally in-semicomplete digraph, it is suffices to show that $D$ satisfies the $\alpha$-property. By Theorem~\ref{lem-arc-resultado}\rm{(ii)}, $V(D)$ admits a partition $(V_1,V_2,V_3)$ as described in the statement. First, we show that if $V_1=\emptyset$, then $D$ satisfies the $\alpha$-property. Next, we show that an extended cycle satisfies the $\alpha$-property (it is analogous to the proof of Lemma~\ref{lema_circuito_estendido}). Finally, we show that if $V_1 \neq \emptyset$, then $D$ satisfies the $\alpha$-property. 

For the next two lemmas we need the following auxiliary lemma.

\begin{lemma}[Freitas and Lee, 2021]
\label{arc-in-circuito}
If $D$ is an arc-locally in-semicomplete digraph, then $D$ contains no induced non-oriented odd cycle of length at least five. 
\end{lemma}

\begin{lemma}
\label{lem_arc_in_berge_v1_empty}
Let $D$ be an arc-locally in-semicomplete digraph such that every proper induced subdigraph of $D$ satisfies the $\alpha$-property. Let $(V_1,V_2,V_3)$ be a partition of $V(D)$ as described in Theorem~\ref{lem-arc-resultado}\rm{(ii)}. If $V_1 = \emptyset$, then $D$ satisfies the $\alpha$-property.
\end{lemma}
 
\begin{proof}
Since $V_1=\emptyset$, it follows by Lemma~\ref{lema_arc-livre-triangulo} that $U(D)$ does not contain a cycle of length three. Note that a blocking odd cycle is a non-oriented odd cycle. So by Lemma~\ref{arc-in-circuito} $D$ contains no blocking odd cycle as an induced subdigraph, and hence, $D \in \D$. Since $D \in \D$, it follows by Theorem~\ref{arc-in-BE} that $D$ satisfies the BE-property, and hence, the $\alpha$-property.
\end{proof}

The next lemma states that if a digraph $D$ is an extended cycle, then $D$ satisfies the $\alpha$-property. We omit its proof since it is analogous to the proof of Lemma~\ref{lema_circuito_estendido}, but we use Lemma~\ref{diper-alpha} instead of Lemma~\ref{diperD-BE}.

\begin{lemma}
\label{lema_circuito_estendido-berge}
If a digraph $D$ is an extended cycle, then $D$ satisfies the $\alpha$-property.
\qed
\end{lemma}

Let $D$ be an arc-locally in-semicomplete digraph. By Theorem~\ref{lem-arc-resultado}\rm{(ii)}, $V(D)$ admits a partition $(V_1, V_2, V_3)$ such that $D[V_1]$ is a semicomplete digraph, $V_1 \mapsto V_2$, $V_1 \Rightarrow V_3$ and $V_2 \Rightarrow V_3$.

\begin{lemma}
\label{lem_arc_in_berge_v1_not_empty}
Let $D$ be an arc-locally in-semicomplete digraph such that every proper induced subdigraph of $D$ satisfies the $\alpha$-property. Let $(V_1,V_2,V_3)$ be a partition of $V(D)$ as described in Theorem~\ref{lem-arc-resultado}\rm{(ii)}. If $V_1 \neq \emptyset$, then $D$ satisfies the $\alpha$-property.
\end{lemma}

\begin{proof}
Let $S$ be a maximum stable set of $D$. The proof is divided into two cases, depending on whether $S \cap V_1=\emptyset$ or $S \cap V_1 \neq \emptyset$. \newline

\textbf{Case 1.} $S \cap V_1 = \emptyset$. Let $D'=D-V_1$. Note that $S$ is maximum in $D'$. By hypothesis, $D'$ is $\alpha$-diperfect. Let $\sP'$ be an $S$-path partition of $D'$. Since $V_2 \Rightarrow V_3$, there exists a path $xPy$ in $\sP'$ such that $x$ in $V_2$. Since $D[V_1]$ is a semicomplete digraph, it follows that $D[V_1]$ is diperfect. By Lemma~\ref{diper-alpha}, $D[V_1]$ satisfies the $\alpha$-property; this implies that there exists a Hamiltonian path $uP'v$ in $D[V_1]$. Since $v \mapsto V_2$, $v \to x$. Let $R=uP'vxPy$ be a path formed by the concatenation of $P'$ and $P$. Thus the collection $(\sP'-P) \cup R$ is an $S$-path partition of $D$. \newline

\textbf{Case 2.} $S \cap V_1 \neq \emptyset$. Since $V_1 \mapsto V_2$, $S \cap V_2 = \emptyset$. Let $Q:=Q[X_1,X_2, \ldots, X_k]$ be the odd extended cycle of length at least five corresponding to $D[V_2]$. Let $x_i \in X_i$ for all $i \in \{1,2,\ldots,k\}$ and let $C=x_1x_2 \ldots x_kx_1$ be a cycle of $D$. Let $D'=D-V(C)$. Since $V(C) \cap S = \emptyset$, $S$ is maximum in $D'$. By hypothesis, $D'$ is $\alpha$-diperfect. Let $\sP'$ be an $S$-path partition of $D'$. The rest of proof is divided into two subcases, depending on whether $V(Q) \neq V(C)$ or $V(Q)=V(C)$. \newline

\textbf{Case 2.1.} $V(Q) \neq V(C)$. First, suppose that there exists a vertex $v_i \in X_i-x_i$ that starts (resp., ends) a path $v_iPw$ (resp., $wPv_i$ ) in $\sP'$ for some $i \in \{1,2,\ldots,k\}$. Let $x_iP'x_{i-1}$ (resp., $x_{i+1}P'x_i$) be a path in $C$ containing $V(C)$. By definition of extended cycle, $x_{i-1} \to v_i$ (resp., $v_i \to x_{i+1}$). Let $R=x_iP'x_{i-1}v_iPw$ (resp., $R=wPv_ix_{i+1}P'x_i$) be a path. Thus the collection $(\sP'-P) \cup R$ is an $S$-path partition of $D$. So we may assume that there exists no vertex in $V(Q)-V(C)$ that starts or ends a path of $\sP'$. Thus there exists a vertex $v_i \in X_i-x_i$ such that $v_i$ is an intermediate vertex in a path $xPy$ of $\sP'$ for some $i \in \{1,2,\ldots,k\}$. Let $w$ be the vertex of $P$ that dominates $v_i$. Let $xP_1w$ and $v_iP_2y$ be the subpaths of $P$. Since $x_i,v_i$ belong to the same $X_i$ of $Q$ and $V_2 \Rightarrow V_3$, it follows that $w \in V_1 \cup X_{i-1}$. Since $V_1 \cup X_{i-1} \mapsto X_i$, $w \to x_i$. By definition of extended cycle, $x_{i-1} \to v_i$. Let $x_iP'x_{i-1}$ be a path in $C$ containing $V(C)$. Let $R=xP_1wx_iP'x_{i-1}v_iP_2y$ be the path formed by inserting $P'$ between $P_1$ and $P_2$. Thus the collection $(\sP'-P) \cup R$ is an $S$-path partition of $D$. \newline

\textbf{Case 2.2.} $V(Q) = V(C)$. Since $D'=D-V(C)$, $V(D')=V_1 \cup V_3$. Since $D[V_1]$ is a semicomplete digraph, $\alpha(D[V_1])=1$. Since $\alpha(D[V(Q)])>1$, $S \cap V_1 \neq \emptyset$ and $S$ is a maximum stable set in $D$, it follows that $V_3 \neq \emptyset$. Recall that $N_d$ is the set of vertices that are at distance $d$ from $Q$. Since $V_1 \mapsto V_2$, $N_d \subseteq V_3$ for $d\geq 1$. By Lemma~\ref{arc-in-F-prop}\rm{(v)}, $N^-(N_1) \subseteq V(Q) \cup V_1$. Assume that there exists a vertex $v$ in $N_1$ such that $v$ starts a path $vPw$ in $\sP'$. Without loss of generality, assume that $x_1 \in V(C)$ dominates $v$ in $D$. Let $x_2P'x_1$ be a path in $C$ containing $V(C)$. Let $R=x_2P'x_1vPw$ be the path formed by the concatenation of $P'$ and $P$. Thus the collection $(\sP'-P) \cup R$ is an $S$-path partition of $D$. So we may assume that there exists no vertex $v$ in $N_1$ such that $v$ starts a path in $\sP'$. Since $N^-(N_1) \subseteq V(Q) \cup V_1$ and $V_1 \Rightarrow V_3$, there exists a path $xPy$ in $ \sP'$ such that $P$ contains vertices $w \in V_1$ and $v \in N_1$ where $w \to v$. Let $xP_1w$ and $vP_2y$ be the subpaths of $P$. Without loss of generality, assume that $x_1 \in V(C)$ dominates $v$ in $D$. Let $x_2P'x_1$ be a path in $C$ containing $V(C)$. Since $V_1 \mapsto V_2$, $w \to x_2$. Let $R=xP_1wx_2P'x_1vP_2y$ be the path formed by inserting $P'$ between $P_1$ and $P_2$. Thus the collection $(\sP'-P) \cup R$ is an $S$-path partition of $D$. This finishes the proof.
\end{proof}

Now, we are ready for the main result of this section.

\begin{theorem}
\label{teo_arc_in_berge}
Let $D$ be an arc-locally in-semicomplete digraph. If $D \in \sB$, then $D$ is $\alpha$-diperfect. 
\end{theorem}

\begin{proof}
Since every induced subdigraph of $D$ is also an arc-locally in-semicomplete digraph, it suffices to show that $D$ satisfies the $\alpha$-property. If $D$ is diperfect or $D$ has a clique cut, then the result follows by Lemmas~\ref{diper-alpha}, \ref{samb_part} and \ref{samb_clique}. So we may assume that $V(D)$ can be partitioned into $(V_1, V_2, V_3)$ as described in Theorem~\ref{lem-arc-resultado}\rm{(ii)}. If $V_1 = V_3 = \emptyset$, then by Lemma~\ref{lema_circuito_estendido-berge} $D$ satisfies the $\alpha$-property. So $V_1 \cup V_3 \neq \emptyset$. If $V_1 = \emptyset$, then the result follows by Lemma~\ref{lem_arc_in_berge_v1_empty}. If $V_1 \neq \emptyset$, then the result follows by Lemma~\ref{lem_arc_in_berge_v1_not_empty}. 
\end{proof}

Similarly to Theorem~\ref{arc-out-BE}, we have the following result.

\begin{theorem}
\label{teo_arc_out_berge}
Let $D$ be an arc-locally out-semicomplete digraph. If $D \in \sB$, then $D$ is $\alpha$-diperfect. 
\qed
\end{theorem}

\section{Conclusion}
\label{conclu}

In this paper, we have shown some structural results for $\alpha$-diperfect digraphs and BE-diperfect digraphs. In particular, the Theorems~\ref{arc-in-maior-S} and \ref{arc-in-maior-S-be} state that if a digraph $D$ is a minimal counterexample to both conjectures, then $\alpha(D) < \frac {\vert V(D)\vert }{2}$. This result suggests that dealing with digraph with small stability number may be the most difficult part of both conjectures. We also have shown that both conjectures hold for arc-locally (out) in-semicomplete digraphs. 

Moreover, Conjectures~\ref{conj_berge} and~\ref{conj_be} are somehow similar to Berge's conjecture on perfect graphs (nowadays known as Strong Perfect Graph Theorem). Furthermore, for more than three decades no results regarding Conjecture~\ref{conj_berge} were published. This suggests that both problems may be very difficult.


\bibliography{ref}


\begin{thebibliography}{15}
\ifx \bisbn   \undefined \def \bisbn  #1{ISBN #1}\fi
\ifx \binits  \undefined \def \binits#1{#1}\fi
\ifx \bauthor  \undefined \def \bauthor#1{#1}\fi
\ifx \batitle  \undefined \def \batitle#1{#1}\fi
\ifx \bjtitle  \undefined \def \bjtitle#1{#1}\fi
\ifx \bvolume  \undefined \def \bvolume#1{\textbf{#1}}\fi
\ifx \byear  \undefined \def \byear#1{#1}\fi
\ifx \bissue  \undefined \def \bissue#1{#1}\fi
\ifx \bfpage  \undefined \def \bfpage#1{#1}\fi
\ifx \blpage  \undefined \def \blpage #1{#1}\fi
\ifx \burl  \undefined \def \burl#1{\textsf{#1}}\fi
\ifx \doiurl  \undefined \def \doiurl#1{\url{https://doi.org/#1}}\fi
\ifx \betal  \undefined \def \betal{\textit{et al.}}\fi
\ifx \binstitute  \undefined \def \binstitute#1{#1}\fi
\ifx \binstitutionaled  \undefined \def \binstitutionaled#1{#1}\fi
\ifx \bctitle  \undefined \def \bctitle#1{#1}\fi
\ifx \beditor  \undefined \def \beditor#1{#1}\fi
\ifx \bpublisher  \undefined \def \bpublisher#1{#1}\fi
\ifx \bbtitle  \undefined \def \bbtitle#1{#1}\fi
\ifx \bedition  \undefined \def \bedition#1{#1}\fi
\ifx \bseriesno  \undefined \def \bseriesno#1{#1}\fi
\ifx \blocation  \undefined \def \blocation#1{#1}\fi
\ifx \bsertitle  \undefined \def \bsertitle#1{#1}\fi
\ifx \bsnm \undefined \def \bsnm#1{#1}\fi
\ifx \bsuffix \undefined \def \bsuffix#1{#1}\fi
\ifx \bparticle \undefined \def \bparticle#1{#1}\fi
\ifx \barticle \undefined \def \barticle#1{#1}\fi
\bibcommenthead
\ifx \bconfdate \undefined \def \bconfdate #1{#1}\fi
\ifx \botherref \undefined \def \botherref #1{#1}\fi
\ifx \url \undefined \def \url#1{\textsf{#1}}\fi
\ifx \bchapter \undefined \def \bchapter#1{#1}\fi
\ifx \bbook \undefined \def \bbook#1{#1}\fi
\ifx \bcomment \undefined \def \bcomment#1{#1}\fi
\ifx \oauthor \undefined \def \oauthor#1{#1}\fi
\ifx \citeauthoryear \undefined \def \citeauthoryear#1{#1}\fi
\ifx \endbibitem  \undefined \def \endbibitem {}\fi
\ifx \bconflocation  \undefined \def \bconflocation#1{#1}\fi
\ifx \arxivurl  \undefined \def \arxivurl#1{\textsf{#1}}\fi
\csname PreBibitemsHook\endcsname

\bibitem{bang2008digraphs}
\begin{bbook}
\bauthor{\bsnm{Bang-Jensen}, \binits{J.}},
\bauthor{\bsnm{{Gutin, Gregory Z.}}}:
\bbtitle{Digraphs: Theory, Algorithms and Applications}.
\bpublisher{Springer},
\blocation{London}
(\byear{2008})
\end{bbook}
\endbibitem

\bibitem{Bondy08}
\begin{bbook}
\bauthor{\bsnm{Bondy}, \binits{J.A.}},
\bauthor{\bsnm{{Murty, U.S.R.}}}:
\bbtitle{Graph Theory}.
\bpublisher{Springer},
\blocation{London}
(\byear{2008})
\end{bbook}
\endbibitem

\bibitem{chudnovsky2006strong}
\begin{botherref}
\oauthor{\bsnm{Chudnovsky}, \binits{M.}},
\oauthor{\bsnm{{Robertson, N.}}},
\oauthor{\bsnm{{Seymour, P.}}},
\oauthor{\bsnm{{Thomas, R.}}}:
The strong perfect graph theorem.
Annals of mathematics,
51--229
(2006)
\end{botherref}
\endbibitem

\bibitem{berge1981}
\begin{bchapter}
\bauthor{\bsnm{Berge}, \binits{C.}}:
\bctitle{Diperfect graphs}.
In: \beditor{\bsnm{Rao}, \binits{S.B.}} (ed.)
\bbtitle{Combinatorics and Graph Theory},
pp. \bfpage{1}--\blpage{8}.
\bpublisher{Springer},
\blocation{Berlin, Heidelberg}
(\byear{1981})
\end{bchapter}
\endbibitem

\bibitem{tesemaycon2018}
\begin{botherref}
\oauthor{\bsnm{Sambinelli}, \binits{M..}},
\oauthor{\bsnm{{Lee, O.}}}:
Partition problems in graphs and digraphs.
PhD thesis,
University of Campinas - UNICAMP
(2018)
\end{botherref}
\endbibitem

\bibitem{ssl}
\begin{botherref}
\oauthor{\bparticle{da} \bsnm{Silva}, \binits{C.N.}},
\oauthor{\bsnm{Lee}, \binits{O.}},
\oauthor{\bsnm{Sambinelli}, \binits{M.}}:
Perfect digraphs.
Submitted.
(2019)
{\href{https://arxiv.org/abs/1904.02799}{{arXiv:1904.02799}}}
{[math.CO]}
\end{botherref}
\endbibitem

\bibitem{hall1935}
\begin{barticle}
\bauthor{\bsnm{Hall}, \binits{P.}}:
\batitle{On representatives of subsets}.
\bjtitle{Journal of the London Mathematical Society}
\bvolume{1}(\bissue{1}),
\bfpage{26}--\blpage{30}
(\byear{1935})
\end{barticle}
\endbibitem

\bibitem{berge1957}
\begin{barticle}
\bauthor{\bsnm{Berge}, \binits{C.}}:
\batitle{Two theorems in graph theory}.
\bjtitle{Proceedings of the National Academy of Sciences of the United States
  of America}
\bvolume{43}(\bissue{9}),
\bfpage{842}
(\byear{1957})
\end{barticle}
\endbibitem

\bibitem{freitas2021}
\begin{botherref}
\oauthor{\bsnm{Freitas}, \binits{L.I.B.}},
\oauthor{\bsnm{Lee}, \binits{O.}}:
Some results on structure of all arc-locally (out) in-semicomplete digraphs.
Submitted.
(2021)
{\href{https://arxiv.org/abs/2104.11019}{{arXiv:2104.11019}}}
{[math.CO]}
\end{botherref}
\endbibitem

\bibitem{alis}
\begin{botherref}
\oauthor{\bsnm{Bang-Jensen}, \binits{J.}}:
Arc-local tournament digraphs: A generalisation of tournaments and bipartite
  tournaments.
Technical report
(1993)
\end{botherref}
\endbibitem

\bibitem{wang2009structure}
\begin{barticle}
\bauthor{\bsnm{Wang}, \binits{S.}},
\bauthor{\bsnm{{Wang, R.}}}:
\batitle{The structure of strong arc-locally in-semicomplete digraphs}.
\bjtitle{Discrete Mathematics}
\bvolume{309}(\bissue{23-24}),
\bfpage{6555}--\blpage{6562}
(\byear{2009})
\end{barticle}
\endbibitem

\bibitem{wang2011}
\begin{barticle}
\bauthor{\bsnm{Wang}, \binits{S.}},
\bauthor{\bsnm{Wang}, \binits{R.}}:
\batitle{Independent sets and non-augmentable paths in arc-locally
  in-semicomplete digraphs and quasi-arc-transitive digraphs}.
\bjtitle{Discrete mathematics}
\bvolume{311}(\bissue{4}),
\bfpage{282}--\blpage{288}
(\byear{2011})
\end{barticle}
\endbibitem

\bibitem{wang2019critical}
\begin{barticle}
\bauthor{\bsnm{Wang}, \binits{R.}}:
\batitle{Critical kernel imperfect problem in generalizations of bipartite
  tournaments}.
\bjtitle{Graphs and Combinatorics}
\bvolume{35}(\bissue{3}),
\bfpage{669}--\blpage{675}
(\byear{2019})
\end{barticle}
\endbibitem

\bibitem{galeana2009}
\begin{barticle}
\bauthor{\bsnm{Galeana-Sánchez}, \binits{H.}},
\bauthor{\bsnm{Goldfeder}, \binits{I.A.}}:
\batitle{A classification of arc-locally semicomplete digraphs}.
\bjtitle{Electronic Notes in Discrete Mathematics}
\bvolume{34},
\bfpage{59}--\blpage{61}
(\byear{2009}).
\doiurl{10.1016/j.endm.2009.07.010}.
\bcomment{European Conference on Combinatorics, Graph Theory and Applications
  (EuroComb 2009)}
\end{barticle}
\endbibitem

\bibitem{galeana2012}
\begin{barticle}
\bauthor{\bsnm{Galeana-Sánchez}, \binits{H.}},
\bauthor{\bsnm{Goldfeder}, \binits{I.A.}}:
\batitle{A classification of all arc-locally semicomplete digraphs}.
\bjtitle{Discrete Mathematics}
\bvolume{312}(\bissue{11}),
\bfpage{1883}--\blpage{1891}
(\byear{2012}).
\doiurl{10.1016/j.disc.2012.02.022}
\end{barticle}
\endbibitem

\end{thebibliography}


\end{document}